\documentclass[11pt,reqno]{amsart}
\usepackage{amsmath,amscd,amssymb,amsfonts,amsthm,courier,relsize,bm}
\usepackage{hyperref,enumitem,mathrsfs,mathtools,slashed}
\usepackage{sidecap}

\usepackage{xcolor}  
\hypersetup{
    colorlinks,
    linkcolor={red!50!black},
    citecolor={blue!70!black},
    urlcolor={blue!80!black}
}

\newtheorem{theorem}{Theorem}[section]
\newtheorem{corollary}[theorem]{Corollary}
\newtheorem{proposition}[theorem]{Proposition}

\newtheorem{lemma}[theorem]{Lemma}
\theoremstyle{definition}    
\newtheorem{definition}[theorem]{Definition}
\theoremstyle{remark}

\newtheorem{remark}[theorem]{Remark}

\newtheorem{example}[theorem]{Example}

\newcommand{\pair}[2]{\langle #1, #2 \rangle}
\newcommand{\ignore}[1]{}

\newcommand{\ol}[1]{\overline{#1}}
\newcommand{\ul}[1]{\underline{#1}}
\newcommand{\ti}[1]{\widetilde{#1}}
\newcommand{\wh}[1]{\widehat{#1}}
\newcommand{\scr}[1]{\mathscr{#1}}

\newcommand{\tn}[1]{\textnormal{#1}}
\newcommand{\mf}[1]{\mathfrak{#1}}
\renewcommand{\sf}[1]{\mathsf{#1}}
\renewcommand{\i}{{\mathrm{i}}}

\def\dirac{\ensuremath{\slashed{\partial}}}
\def\d{\ensuremath{\mathrm{d}}}

\def\ad{\ensuremath{\textnormal{ad}}}
\def\g{\ensuremath{\mathfrak{g}}}

\def\t{\ensuremath{\mathfrak{t}}}

\def\h{\ensuremath{\mathfrak{h}}}

\def\C{\ensuremath{\mathcal{C}}}

\def\A{\ensuremath{\mathcal{A}}}
\def\P{\ensuremath{\mathcal{P}}}
\def\I{\ensuremath{\mathcal{I}}}

\def\S{\ensuremath{\mathcal{S}}}
\def\Q{\ensuremath{\mathcal{Q}}}
\def\L{\ensuremath{\mathcal{L}}}
\def\R{\ensuremath{\mathcal{R}}}
\def\B{\ensuremath{\mathcal{B}}}
\def\E{\ensuremath{\mathcal{E}}}
\def\D{\ensuremath{\mathcal{D}}}

\def\bC{\ensuremath{\mathbb{C}}}
\def\bR{\ensuremath{\mathbb{R}}}
\def\bZ{\ensuremath{\mathbb{Z}}}

\def\bQ{\ensuremath{\mathbb{Q}}}

\def\Hom{\ensuremath{\textnormal{Hom}}}
\def\ker{\ensuremath{\textnormal{ker}}}
\def\supp{\ensuremath{\textnormal{supp}}}

\def\Cl{\ensuremath{\textnormal{Cl}}}
\def\Ahat{\ensuremath{\widehat{\tn{A}}}}
\def\DH{\ensuremath{\tn{DH}}}
\def\Ch{\ensuremath{\tn{Ch}}}

\DeclareMathOperator*{\Star}{\star}

\def\dim{\ensuremath{\textnormal{dim}}}
\def\Sym{\ensuremath{\textnormal{Sym}}}
\def\Pol{\ensuremath{\textnormal{Pol}}}
\def\pol{\ensuremath{\textnormal{pol}}}
\def\QPol{\ensuremath{\tn{QPol}}}

\begin{document}
\title[Semi-classical analysis of piecewise quasi-polynomial functions]{Semi-classical analysis of piecewise quasi-polynomial functions and applications to geometric quantization}
\author{Yiannis Loizides} \address{Cornell University} \email{yl3542@cornell.edu}
\author{Paul-Emile Paradan} \address{Universit{\'e} de Montpellier} \email{paradan@math.univ-montp2.fr}
\author{Michele Vergne} \address{Institut de Math{\'e}matiques de Jussieu - Paris Rive Gauche} \email{michele.vergne@imj-prg.fr}

\sloppy

\begin{abstract}
Motivated by applications to multiplicity formulas in index theory, we study a family of distributions $\Theta(m;k)$ associated to a piecewise quasi-polynomial function $m$.  The family is indexed by an integer $k \in \bZ_{>0}$, and admits an asymptotic expansion as $k \rightarrow \infty$, which generalizes the expansion obtained in the Euler-Maclaurin formula.  When $m$ is the multiplicity function arising from the quantization of a symplectic manifold, the leading term of the asymptotic expansion is the Duistermaat-Heckman measure. Our main result is that $m$ is uniquely determined by a collection of such asymptotic expansions. We also show that the construction is compatible with pushforwards. As an application, we describe a simpler proof that formal quantization is functorial with respect to restrictions to a subgroup.
\bigskip

\noindent \emph{To the memory of Hans Duistermaat.}
\end{abstract}

\maketitle

\section{Introduction}
Let $V$ be a finite dimensional real vector space equipped with a lattice $\Lambda$.  Let $P \subset V$ be a rational polyhedron.  The Euler-Maclaurin formula (\cite{BerlineVergneLocalAsym, GuilleminSternbergEulerMaclaurin}) gives an asymptotic estimate, when $k \rightarrow \infty$, for the Riemann sum
\[ \sum_{\lambda \in kP \cap \Lambda} \varphi \big(\tfrac{\lambda}{k}\big) \]
of the values of a test function $\varphi$ at the sample points $k^{-1}\Lambda \cap P$ of $P$, with leading term
\[ k^{\dim(P)}\int_P \varphi \, d\mu_\Lambda,\]
where $\mu_\Lambda$ is the translation-invariant measure on $V$ such that the volume of a fundamental domain for the $\Lambda$ action is $1$.  One may consider the slightly more general case of a weighted sum.  Let $q(k,\lambda)$ be a quasi-polynomial function on $\bZ \oplus \Lambda$.  We consider, for $k \ge 1$, the distribution (Figure \ref{fig:intropic})
\[ \pair{\Theta(P,q;k)}{\varphi}=\sum_{\lambda \in kP\cap \Lambda}q(k,\lambda)\varphi \big(\tfrac{\lambda}{k}\big) \]
and we show that the function $k \mapsto \pair{\Theta(P,q;k)}{\varphi}$ admits an asymptotic expansion when $k \rightarrow \infty$ in powers of $k^{-1}$ with the coefficients being periodic functions of $k$.

\begin{figure}
{\centering
\includegraphics[clip, trim=2.5cm 19.5cm 2.5cm 4cm, width=1\textwidth]{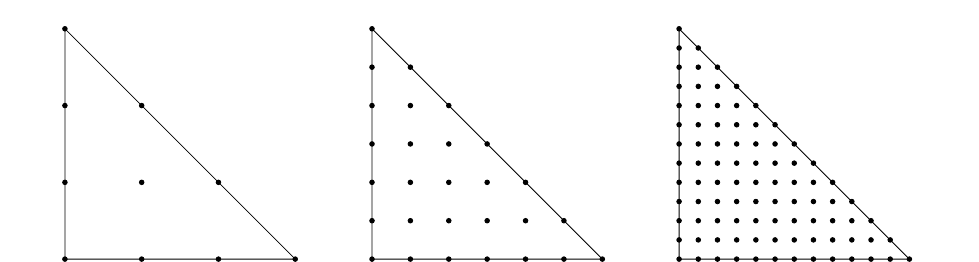}
}
\caption{Support of $\Theta(P,q;k)$, $k=3,6,12$, with $\Lambda=\bZ^2 \subset \bR^2=V$, and $P=\{x\ge 0,y\ge 0, x+y \le 1\}$.}
\label{fig:intropic}
\end{figure}

We further extend this result to a space $\S(\Lambda)$ of piecewise quasi-polynomial functions on $\bZ \oplus \Lambda \subset \bR \oplus V$.  Denote by $C_{P,\sigma}$ the affine cone $\{(t,tv+\sigma)|t \in (0,\infty),\, v \in P\}$ generated by a rational polyhedron $P$ in $V$ (placed at level $1$ in $\bR \oplus V$) and translated by a rational element $\sigma \in V$.  Let
\[ m(P,\sigma,q)(k,\lambda)=q(k,\lambda)[C_{P,\sigma}\cap(\bZ \oplus \Lambda)](k,\lambda), \]  
where $[C_{P,\sigma}\cap(\bZ \oplus \Lambda)]$ denotes the characteristic function of $C_{P,\sigma}\cap(\bZ \oplus \Lambda)$ and $q(k,\lambda)$ is a quasi-polynomial function.  The space $\S(\Lambda)$ consists of possibly infinite linear combinations of functions $m(P,\sigma,q)$ on $\bZ\oplus \Lambda$, subject to a local finiteness condition (see Definition \ref{def:slambda}).  Although the function $k \mapsto m(k,k\lambda)$ is not, in general, a quasi-polynomial function of $k\ge 1$ (see Figure \ref{fig:intropic2}), an important feature of functions $m \in \S(\Lambda)$ is that they are completely determined by their behavior for large $k$, as we prove in Corollary \ref{cor:largek}.

A motivating example comes from geometry: let $M$ be a smooth complex projective variety with an action of a torus $T$, let $\L$ be the corresponding $T$-equivariant ample line bundle, and let $\mathcal{E}$ be an auxiliary $T$-equivariant vector bundle.  For $t \in T$, define $m_{\mathcal{E}}$ by the equation
\begin{equation} 
\label{eqn:eulerchar}
\sum_{i=0}^{\dim(M)}(-1)^i \tn{Tr}\Big(t,H^i\big(M,\mathcal{O}(\mathcal{L}^k\otimes \mathcal{E})\big)\Big)=\sum_{\lambda \in \Lambda}m_{\mathcal{E}}(k,\lambda)t^\lambda \end{equation}
where $\Lambda$ is the lattice of characters of $T$.  The Atiyah-Bott formula shows that $m_{\mathcal{E}}(k,\lambda)$ is piecewise quasi-polynomial for a finite set of cones $C_{P,\sigma}$.  Moreover if $\mathcal{E}=M \times \bC$ is the trivial vector bundle, then one may take all shifts $\sigma=0$. For $m=m_{M\times \bC}$, the asymptotic expansion of the distribution
\[ \sum_{\lambda \in \Lambda} m(k,\lambda)\delta_{\lambda/k} \]
is an important object associated to $M$, involving the Duistermaat-Heckman measure and the Todd class of $M$; see \cite{VergneGradedTodd} for a study of its asymptotics.  

The determination of similar asymptotics in the more general case of twisted Dirac operators  is carried out by two of the authors in the article \cite{ParVerSemiclassical}, where some of the results established here, together with the $[Q,R]=0$ theorem for formal quantization, are used to give a proof of the functoriality under restriction to subgroups of the formal quantization of a possibly non compact symplectic manifold \cite{ParadanFormalI}, or more generally of a spin-c manifold \cite{ParadanFormalIII}. Using the general results established here, we sketch a simpler proof of functoriality (Section \ref{sec:motivation}) that does not rely on the $[Q,R]=0$ theorem, and therefore generalizes to formal indices of a more general class of vector bundles (which need not satisfy $[Q,R]=0$).

\begin{figure}
{\centering
\includegraphics[clip, trim=0cm 14.5cm 1cm 4cm, width=1\textwidth]{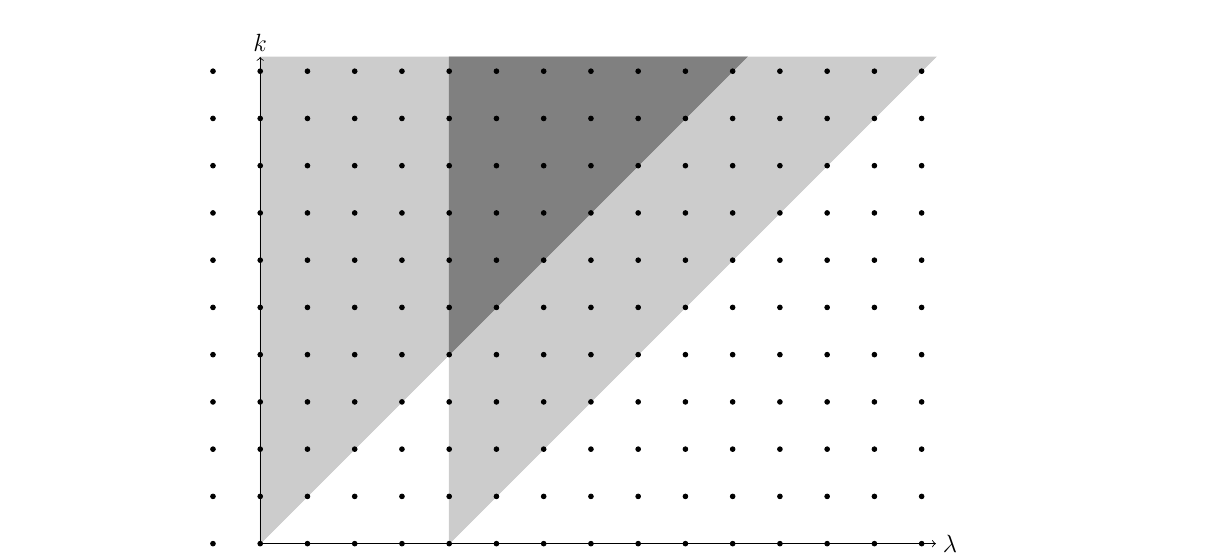}
}
\caption{The function $m=[C_P]+[C_{P,\sigma}]$ for $\Lambda=\bZ\subset \bR=V$, $P=[0,1]$, and $\sigma=4$.  On the light gray region $m=1$, and on the dark gray region $m=2$.}
\label{fig:intropic2}
\end{figure}

Let us give a brief summary of the contents of the article.  In Section 2, which is independent of the rest of the article, we sketch the proof of functoriality under restriction to subgroups for formal equivariant indices; this provides the geometric motivation for the main results of the article. For $m$ a function in the space $\S(\Lambda)$ mentioned above, we consider the family of distributions, parametrized by $k \in \bZ_{>0}$,
\[ \Theta(m;k)=\sum_{\lambda \in \Lambda} m(k,\lambda)\delta_{\lambda/k},\]
and prove (Theorem \ref{thm:AsymExp}) that this family admits an asymptotic expansion $\A(m;k)$, as $k \rightarrow \infty$, in powers of $k^{-1}$ and with coefficients that are periodic in $k$ (locally in $V$).  We compute the leading term and state conditions on the types of distributions which can occur in the expansion. In particular, the coefficient of $k^n$ in $\A(m;k)$ is a distribution on $V$ which restricts to a piecewise polynomial measure on  the complement of an arrangement of hyperplanes. A striking example of such a piecewise polynomial measure is the Duistermaat-Heckman measure arising as the leading term in $\A(m;k)$, where $m$ is the multiplicity function $m_{\mathcal{E}}(k,\lambda)$ (\ref{eqn:eulerchar}).

We study the extent to which the map $m \mapsto \A(m)$ preserves information about $m$.  The map $\A$ has a non-trivial kernel: in $1$-dimension with $k,\lambda \in \bZ$ the function $m(k,\lambda)=(-1)^\lambda$ is quasi-polynomial and $\A(m)=0$.  Let $\Lambda^\wedge$ be the group of characters of $\Lambda$ (a compact torus), and $\Lambda^\wedge_{\bQ}$ the subgroup of elements of finite order.  If $g \in \Lambda^\wedge_{\bQ}$ then $(g\cdot m)(k,\lambda):=g^\lambda m(k,\lambda)$ is again an element of $\S(\Lambda)$.  The following unicity result is our Theorem \ref{thm:unicity}.
\begin{theorem}
\label{thm:intro}
If $\A(g\cdot m)=0$ for all $g \in \Lambda^\wedge_\bQ$ then $m=0$.
\end{theorem}
\noindent Thus the piecewise quasi-polynomial function $m$ is entirely determined by the collection of asymptotic expansions $\A(g\cdot m)$, $g \in \Lambda^\wedge_{\bQ}$.  There is a slightly better behaved but still large subset $\S_{\tn{pol}}(\Lambda)\subset \S(\Lambda)$ consisting of piecewise quasi-polynomial functions admitting a decomposition with some polarization properties (Definition \ref{def:polcones}).  In particular $\S_{\tn{pol}}(\Lambda)$ contains any finite linear combination of the functions $m(P,\sigma,q)$ above.  For $0\ne \eta \in \Lambda$ and $1\ne \zeta \in U(1)_\bQ$ a root of unity, define the finite difference operator
\[(\nabla_\eta^\zeta m)(k,\lambda)=m(k,\lambda)-\zeta\cdot m(k,\lambda-\eta).\]
For $m \in \S_{\tn{pol}}(\Lambda)$ we prove (Corollary \ref{cor:kernelofA}) the following refinement of Theorem \ref{thm:intro}.

\begin{theorem}
\label{thm:introkernel}
The kernel of $\A$ in $\S_{\tn{pol}}(\Lambda)$ consists of locally finite sums of functions $m$ such that there exists $0\ne \eta \in \Lambda$, $1\ne \zeta \in U(1)_\bQ$, and an $N \in \bZ$ with $(\nabla_\eta^\zeta)^N m=0$.
\end{theorem}
\noindent The equation $(\nabla_\eta^\zeta)^Nm=0$ means that there is a direction $\eta$ such that $m(k,\lambda)$ is quasi-polynomial in the direction $\eta$, but not polynomial (as $\zeta \ne 1$).  We conjecture that the collection of functions mentioned in the theorem is in fact the kernel of $\A$ on the full space $\S(\Lambda)$.

A simple corollary of Theorem \ref{thm:introkernel} is the fact that $\A(m)=0$ implies that $m=0$ whenever the support of $m$ does not contain any line, for example, if $m$ is a finite linear combination of functions $m(P,\sigma,q)$ where the $P$ are polytopes.  This case is important since it occurs, for example, when studying multiplicity functions associated to torus actions on smooth complex projective varieties as in \eqref{eqn:eulerchar}, and also in the slightly more general setting where one allows $\L^k$ to be replaced by $\L^k \otimes \mathcal{E}$, for an auxiliary $T$-equivariant holomorphic vector bundle $\mathcal{E}$.

We prove that the distributions $\Theta(m;k)$ and series $\A(m;k)$ behave functorially under pushforwards with proper supports (Corollary \ref{cor:pushfunctorial}), that is, if $\pi \colon V \rightarrow V^\prime$ is a surjective linear map with rational kernel and satisfies a properness condition with respect to $m$, then there is a piecewise quasi-polynomial function $\pi_\ast m \in \S(\pi(\Lambda))$ such that 
\[\pi_\ast \Theta(m;k)=\Theta(\pi_\ast m;k)\sim \pi_\ast \A(m;k).\]

The results of this article on injectivity of asymptotic expansions overlap with an earlier article by two of the authors (P. and V. \cite{ParVerAsymptotic}), which will not be submitted for publication.  Since the simpler case considered in \emph{loc. cit.}  is already of interest and may be easier to follow, we have left this preliminary version on ArXiv, for a first approach.  As already mentioned, a new feature of the present version is that we allow for shifted cones $C_{P,\sigma}$ ($\sigma$ possibly non-zero) systematically throughout.  Furthermore, in this new version we give a conjectural description of the kernel of the map $\A$ (Theorem \ref{thm:introkernel}) and prove it for an important subcase.

Let us give a few small examples to illustrate the results.  The function of $k,\lambda \in \bZ$ given by $m_1(k,\lambda)=1$ if $0\le \lambda \le k$ and $m_1(k,\lambda)=0$ otherwise is piecewise quasi-polynomial in our sense.  The associated family of distributions is
\[ \Theta(m_1;k)=\sum_{\lambda=0}^{k} \delta_{\lambda/k}.\]
Using the Euler-Maclaurin formula, this family admits an asymptotic expansion
\begin{equation} 
\label{eqn:distexp}
\Theta(m_1;k)\sim \A(m_1;k)=k \mu_{[0,1]}+\frac{1}{2}(\delta_0+\delta_1)+k\sum_{n=2}^\infty \frac{B_n}{n!k^n}(-1)^{n-1}(\delta^{(n-1)}_1-\delta^{(n-1)}_0),
\end{equation}
where $\mu_{[0,1]}$ denotes standard Lebesgue measure on $[0,1]$, $B_n$ is the $n^{th}$ Bernoulli number, and $\delta_r^{(n-1)}$ denotes the $(n-1)^{st}$ derivative of the Dirac delta distribution $\delta_r(x)=\delta(x-r)$.  Equation \eqref{eqn:distexp} just means that for any test function $\varphi \in C^\infty_c(\bR)$ we have an asymptotic expansion,
\[ \pair{\Theta(m_1;k)}{\varphi}=\sum_{\lambda=0}^{k}\varphi\big(\tfrac{\lambda}{k}\big)\sim k \int_0^1 \varphi(x)dx+\frac{\varphi(0)+\varphi(1)}{2}+k\sum_{n=2}^\infty \frac{B_n}{n!k^n}\big(\varphi^{(n-1)}(1)-\varphi^{(n-1)}(0)\big),\]
and this is the usual Euler-Maclaurin formula.

For a related example, consider $m_2(k,\lambda)=1$ if $2\le \lambda \le k+2$ and $m_2(k,\lambda)=0$ otherwise.  Then $m_2$ is a translate of $m_1$.  In this case notice that the corresponding distribution
\[ \Theta(m_2;k)=\sum_{\lambda=2}^{k+2} \delta_{\lambda/k} \]
does not have its support contained in $[0,1]$.  It follows from Taylor's theorem that
\[ \Theta(m_2;k)\sim \A(m_2;k)=\sum_{j\ge 0}\frac{(-2)^j}{k^j j!}\partial^j \A(m_1;k), \]
where $\partial=\d/\d x$, that is, we apply the formal series of differential operators $e^{-2\partial/k}$ to $\A(m_1;k)$.  Notice that $\A(m_2;k)$ has support contained in $[0,1]$, and differs from $\A(m_1;k)$ at all orders below the leading order.

For an example of higher degree, consider $m_3(k,\lambda)=(\lambda+1)m_1(k,\lambda)$.  Then
\[ \Theta(m_3;k,x)=\sum_{\lambda=0}^{k-1}(\lambda+1)\delta_{\lambda/k}(x)=(kx+1)\Theta(m_1;k,x),\]
and hence
\[ \Theta(m_3;k,x)\sim \A(m_3;k,x)=(kx+1)\A(m_1;k,x).\]
For instance, the leading and sub-leading terms are, respectively,
\begin{equation} 
\label{eqn:sublead}
k^2 x \mu_{[0,1]}, \qquad \qquad k\mu_{[0,1]}+\tfrac{1}{2}k\delta_1.
\end{equation}

Consider the following 2-dimensional example.  For $k,\lambda_1,\lambda_2 \in \bZ$ let $m(k,\lambda_1,\lambda_2)=1$ if $\lambda_1 \ge 0$, $\lambda_2 \ge 0$, $\lambda_1+\lambda_2 \le k$, and $m(k,\lambda_1,\lambda_2)=0$ otherwise.  The support of $\Theta(m;k)$ for $k=3,6,12$ is shown in Figure \ref{fig:intropic}.  Let $\pi \colon \bR^2 \rightarrow \bR^2/W$ be the quotient map for the subspace $W=\tn{span}\{(1,-1)\}$.  We identify the quotient with $\bR$ in such a way that the quotient lattice $\pi(\bZ^2) \subset \bR^2/W$ is identified with $\bZ$.  Then $(\pi_\ast m)(k,\lambda)=\lambda+1$ if $0\le \lambda \le k$ and $(\pi_\ast m)(k,\lambda)=0$ otherwise, and so $\pi_\ast m=m_3$ (see preceding paragraph).  Let $P=\{(\lambda_1,\lambda_2)|\lambda_1,\lambda_2 \ge 0, \lambda_1+\lambda_2\le 1 \}$, and let $\partial_0 P=P\cap \{\lambda_1+\lambda_2=1\}$, $\partial_1 P=P\cap \{\lambda_2=0\}$, $\partial_2 P=P\cap \{\lambda_1=0\}$ be the closed 1-dimensional faces of $P$.  The lattice $\Lambda=\bZ^2$ determines a canonical normalization of the Lebesgue measure on any affine rational subspace, and hence also (by restriction) on any rational convex polyhedron $Q$; we write $\mu_Q$ for the corresponding distribution supported on $Q$.  The leading and sub-leading terms in the asymptotic expansion of $\Theta(m;k)$ are, respectively,
\[ k^2 \mu_P, \qquad \qquad \tfrac{1}{2}k(\mu_{\partial_1P}+\mu_{\partial_2P}+\mu_{\partial_0P}).\]
Pushing forward using $\pi_\ast \mu_P=x\mu_{[0,1]}$, $\pi_\ast \mu_{\partial_1P}=\mu_{[0,1]}=\pi_\ast \mu_{\partial_2P}$, $\pi_\ast \mu_{\partial_0P}=\delta_1$ we obtain the leading and sub-leading terms in the asymptotic expansion of $\pi_\ast \Theta(m;k)=\Theta(\pi_\ast m;k)$,
\[ k^2 x\mu_{[0,1]}, \qquad \qquad k\mu_{[0,1]}+\tfrac{1}{2}k\delta_1.\]
This agrees with the leading and sub-leading terms in the expansion of $\Theta(m_3;k)$, see \eqref{eqn:sublead}.

Let us mention a few questions along with one further instructive example. We wonder whether there are still larger spaces of functions $m(k,\lambda)$ on a lattice $\bZ \oplus \Lambda$ leading to an asymptotic expansion $\A(m)$ and a unicity theorem. Furthermore, it would be interesting to prove the injectivity of the map $m\mapsto \A(m)$ by a general reconstruction theorem. For example, when $m$ is a Kostant partition function associated to a unimodular set of vectors, the Dahmen-Micchelli limit formula \cite{DahmenMicchelli} allows us to reconstruct $m$ from the piecewise polynomial functions associated to $\A(m)$. However, in the general case, the singular part of $\A(m)$ is usually necessary to reconstruct $m$. Here is a simple, but instructive, example. Take $\Lambda=\bZ$, and  $m(k,\lambda)=v(\lambda)$ where $v(\lambda)$ is an arbitrary function on $\bZ$ with finite support. The function $m$ is in our space $\S(\Lambda)$. If $\phi$ is a smooth function on $\bR$, using its Taylor series expansion, we see that $\sum_\lambda v(\lambda)\phi(\lambda/k)$ is asymptotic to $ \sum_{n\geq 0}k^{-n} \frac{c_n(v)}{n!} \phi^{(n)}(0)$, where $c_n(v)=\sum_{\lambda} v(\lambda)\lambda^n$. One then reconstructs the function $v$ from the collection of numbers $c_n(v)$ using the Vandermonde determinant. In particular, the problem of reconstructing $m$ from its asymptotic expansion overlaps with the reconstruction of a function $m$ on a lattice from its moments.

\bigskip

\noindent \textbf{Some conventions and notation.}  Let $V$ be a finite dimensional real vector space.  For any subset $S \subset V$, $\tn{aff}(S)$ denotes the smallest affine subspace containing $S$, and $\tn{lin}(S)$ the linear subspace parallel to $\tn{aff}(S)$.  For $\sigma \in V$, we use the notation $\tau_\sigma$ for translation of functions (or distributions) by $\sigma \in V$, that is, $(\tau_\sigma f)(v)=f(v-\sigma)$.  The characteristic/indicator function of $S \subset V$ will be denoted $[S]$.  Let $\Lambda \subset V$ be a full rank lattice with dual lattice $\Lambda^\ast=\Hom_\bZ(\Lambda,\bZ)$.  Lebesgue measure on $V$ is normalized such that the volume of a fundamental domain for the $\Lambda$ action is $1$.  Using the measure, we identify distributions with generalized functions; given a distribution $\Theta$ on $V$, we will also write $\Theta(v)$ for the corresponding generalized function of $v \in V$.  In this article a `polyhedron' will mean a convex rational polyhedron, that is, a finite intersection of closed half spaces of the form $\{v \in V|\pair{a}{v}\ge c\}$ where $a \in \Lambda^\ast \otimes \bQ$, $c \in \bQ$.  An affine subspace $S$ is rational if and only if $S \cap (\Lambda \otimes \bQ)\ne \emptyset$ and $\tn{lin}(S)\cap \Lambda$ has full rank in $\tn{lin}(S)$.  We write $\Lambda^\wedge=\Hom(\Lambda,U(1))\simeq V^\ast/\Lambda^\ast$ for the group of characters of $\Lambda$ (a compact torus), and $\Lambda^\wedge_{\bQ}$ for the subgroup of elements of finite order. Consistent with this notation, we write $U(1)_\bQ$ for the set of all roots of unity.

\section{Geometric motivation}\label{sec:motivation}
In this section, which is independent from the rest of the article, we describe the geometric motivation for our results (especially Theorem \ref{thm:unicity}) and we sketch a proof of the functoriality-under-restriction property for formal equivariant indices. Let us briefly explain the difficulty encountered for defining formal equivariant indices and their functoriality with respect to subgroups. 
 
Let $G$ be a compact connected Lie group. If $D$ is a $G$-equivariant elliptic operator on a compact manifold $M$, one can define its equivariant index: a function $\tn{index}_G(D)(g)$ on $G$. The Fourier transform  of $\tn{index}(D)(g)$ gives us a multiplicity function $m_G(\lambda)$  on $\hat G$ with finite support. Obviously, when $G'$ is a compact subgroup of $G$, the function $\tn{index}_G(D)$ on $G$ restricts to the function  $\tn{index}_{G'}(D)$ on $G'$. So the multiplicity function $m_{G'}$ is easily computed from $m_G$. Consider $M$ an even dimensional compact manifold, and let $\dirac$ be the Dirac operator  associated to
any graded Clifford module $\E$ on $M$. Let $\L$ be an auxiliary $G$-equivariant line bundle. If we twist $\dirac$ by the powers $\L^k$,
we obtain a family $m_G(k,\lambda)$ of functions on $\hat G$ indexed by $k\in \bZ$.

On a non-compact $G$-manifold  $M$, one  needs additional data to  define a meaningful $G$-equivariant index of the operator $\dirac$ twisted by $\L^k$ and its multiplicity function $m_G(k,\lambda)$. Essentially, one needs a proper moment map $\mu_\g\colon M\to \g^*$ associated to a connection on $\L$ to construct a Fredholm deformation of $\dirac$. This deformation is strongly dependant on the group $G$, so it is not immediate that the formal $G$-index of $\dirac^{\L^k}$ restricted to $G'$ is the formal  $G'$-index. In contrast, the Duistermaat-Heckman measure associated to $\mu_\g$ as well as the other distributions involved in $\A(m_G)$  are naturally functorial with respect to pushforward. So the injectivity of the map $m\mapsto \A(m)$ and the fact that the pushforward of the distribution $\A(m_G)$ is the distribution $\A(m_{G'})$ allows us to conclude that the formal index is functorial with respect to restriction to subgroups. We only sketch the main lines of this argument, since some of the arguments are very similar to \cite{ParVerSemiclassical, VergneGradedTodd, VergneFormalAhat}. However, let us emphasize that since we consider Dirac operators twisted by arbitrary vector bundles, we cannot use any geometric description of $m_G$ such as the one given by the $[Q,R]=0$ theorem. In fact it is easier to directly use the injectivity of the asymptotic expansion to compare the discrete setting with the continuous setting and to prove functoriality, and we do this in this section.

\subsection{Formal indices.}
Let $M$ be an oriented even dimensional Riemannian manifold with an isometric action of a compact connected Lie group $G$. Fix an invariant inner product on the Lie algebra $\g$, which we use to identify $\g \simeq \g^*$. For $X \in \g$, let $X_M \in \mf{X}(M)$ denote the induced vector field on $M$.  Let $(\L,\nabla^\L)$ be an auxiliary $G$-equivariant line bundle with connection. Let $\omega=(\i/2\pi)(\nabla^\L)^2$ be the first Chern form and define the moment map $\mu_\g$ associated to the connection $\nabla^\L$ by
\[ 2\pi \i \pair{\mu_\g}{X}=\sf{L}^\L_X-\nabla^\L_{X_M}, \quad X \in \g,\]
where $\sf{L}^\L_X$ denotes the infinitesimal action of $X$ on $\Gamma(\L)$. The equivariant 2-form $\omega_\g(X):=\omega+2\pi \i\pair{\mu_\g}{X}$ is closed with respect to the differential $\d+2\pi \i \iota(X_M)$.

The \emph{Kirwan vector field} $\kappa \in \mf{X}(M)^G$ is defined by
\[ \kappa(m)=(\mu_\g(m))_M(m), \]
where $\mu_\g(m)\in \g^*$ is identified with an element of $\g$ using the invariant inner product. Fix a maximal torus $T$ with Lie algebra $\t$, as well as a choice of positive chamber $\t^*_+$. The vanishing locus of $\kappa$ is
\[ Z_G=\bigcup_{\beta \in \B} Z_{G,\beta}, \qquad Z_{G,\beta}=G\cdot Z_\beta, \qquad Z_\beta=M^\beta \cap \mu_\g^{-1}(\beta), \]
and $\B$ is the set of $\beta \in \t^*_+$ such that $Z_\beta\ne \emptyset$.

Let $(\E,\nabla^\E)$ be a $G$-equivariant Clifford module bundle with connection. Let $\dirac$ be the Dirac operator on $\E$ defined by the connection, and let $\dirac^{\L}$ be the Dirac operator twisted by $\L$.  Let $\sigma^\L(\xi)$ denote the symbol of $\dirac^\L$. If $M$ is compact then the $G$-equivariant index
\[ Q_G(M,\L,\E)=\tn{index}_G(\dirac^\L) \in R(G) \]
is defined. The non-abelian localization formula in K-theory (\cite{ParadanRiemannRoch, WittenNonAbelian}) then states that
\begin{equation}
\label{eqn:nonabel}
Q_G(M,\L,\E)=\sum_{\beta \in \B} Q_G(M,\L,\E,Z_{G,\beta})
\end{equation}
where $Q_G(M,\L,\E,Z_{G,\beta}) \in R^{-\infty}(G)$ is the index of a transversally elliptic operator on an open neighborhood of the component $Z_{G,\beta}$ of $Z_G$, with symbol given by the restriction of $\sigma^\L(\xi-\kappa)$.

If $M$ is non-compact then the index of $\dirac^\L$ may not be defined. However if $\mu_\g$ is \emph{proper} then, in particular, $\B$ is discrete and each component $Z_{G,\beta}$ is compact, so that there is a chance that the right hand side of \eqref{eqn:nonabel} makes sense, and can be taken as the \emph{definition} of $Q_G(M,\L,\E)$. Indeed this is the case under an additional condition on $\E$ that we will come to shortly. We then think of $Q_G(M,\L,\E)$ as the `formal equivariant index' of $\dirac^\L$. We refer the reader to \cite{MaZhangTransEll,ParadanFormalI,ParadanFormalII,ParadanFormalIII} and references therein for further background on formal equivariant indices.

\subsection{Anti-symmetric multiplicity functions.}
Let $\Lambda \subset \t^*$ denote the (real) weight lattice of $T$. Let $\mf{R}=\mf{R}_+\cup \mf{R}_-\subset \Lambda$ denote the roots. Let $\Lambda_+\subset \Lambda$ denote the dominant weights. Let $\rho$ denote the half-sum of the positive roots. An element $Q_G \in R^{-\infty}(G)$ is uniquely determined by its multiplicity function $m_G \colon \Lambda_+ \rightarrow \bZ$. Any function $m_G\colon \Lambda_+ \rightarrow \bZ$ has a unique extension $m\colon \Lambda \rightarrow \bZ$ which is anti-symmetric for the $\rho$-shifted action of the Weyl group $W=N_G(T)/T$:
\begin{equation} 
\label{eqn:shiftedanti}
m(w\bullet \lambda)=(-1)^{l(w)}m(\lambda), \qquad w\bullet \lambda=w(\lambda+\rho)-\rho.
\end{equation}
Let $Q \in R^{-\infty}(T)$ denote the formal character associated to $m$. There is an induced $\rho$-shifted action of $W$ on $R^{-\infty}(T)$ and $w \bullet Q=(-1)^{l(w)}Q$. For $Q_G \in R(G)$ one has $Q=\Delta \cdot Q_G|_T \in R(T)$ where $\Delta=\prod_{\alpha \in \mf{R}_-}(1-t^\alpha)$.

Let us return to the case of a compact $M$, the character $Q_G(M,\L,\E) \in R(G)$, and the corresponding anti-symmetric character $Q(M,\L,\E) \in R(T)$. There is an equivalent formula to \eqref{eqn:nonabel} in terms of $Q(M,\L,\E)$:
\begin{equation} 
\label{eqn:antinonabel} 
Q(M,\L,\E)=\sum_{\beta \in \B} \sum_{w \in W} (-1)^{l(w)}w\bullet Q(M,\L,\E,Z_\beta).
\end{equation}
To describe the contributions $Q(M,\L,\E,Z_\beta)$ in more detail, let $N_\beta$ be the normal bundle to the fixed-point set $M^\beta$ in $M$. The element $\beta \in \t^*\simeq \t$ acts by a non-degenerate anti-symmetric transformation $A_\beta$ in the fibres of $N_\beta$, hence $N_\beta$ acquires a complex structure $J_\beta=A_\beta/|A_\beta|$. We may define the $\Cl(N_\beta)$-module $\wedge N_\beta^{0,1}$, and the induced Clifford module bundle $\E_\beta$ over $M^\beta$:
\[ \E_\beta=\Hom_{\Cl(N_\beta)}(\wedge N_\beta^{0,1},\E).\]
Then
\begin{equation} 
\label{eqn:antinonabelterm}
Q(M,\L,\E,Z_\beta)=\tn{index}_T\big(\sigma^\L_{Z_\beta}\otimes \Sym(N_\beta^{1,0})\big),
\end{equation}
where $\sigma^\L_{Z_\beta}$ is a $T$-transversally elliptic symbol on $M^\beta$ supported on $Z_\beta$, acting on sections of the $\bZ_2$-graded vector bundle $\L\otimes \E_\beta \wh{\otimes}\! \wedge \! \t^\perp_-$, where $\t^\perp_-$ denotes the orthogonal complement $\t^\perp$ viewed as a complex representation of $T$, with the complex structure determined by the negative roots $\mf{R}_-$. In brief $\sigma^\L_{Z_\beta}$ is the product of a Kirwan-deformed symbol on a neighborhood $U_\beta$ of $Z_\beta$ in $M^\beta$, and the pull-back of a Bott-Thom symbol for $\t^\perp_-$, the latter being included in order to implement multiplication by $\Delta$, and to ensure compactness of the support of the symbol.

Let us now return to the case of a possibly non-compact $M$, assuming $\mu_\g$ is proper. The element $\beta \in \t^*\simeq \t$ acts on the fibres of $\E_\beta$ by a skew-Hermitian linear transformation $E_\beta$ with locally constant pure imaginary eigenvalues, and which also commutes with the $\Cl(TM^\beta)$ action. Let $e_{\tn{min}}(\beta) \in \bR$ denote the least eigenvalue of $\tfrac{1}{\i}E_\beta$ over all components of the compact subset $Z_\beta$. Suppose that
\begin{equation} 
\label{eqn:condition}
|\beta|^2+e_{\tn{min}}(\beta) \xrightarrow{|\beta|\rightarrow \infty} +\infty 
\end{equation}
as $\beta$ ranges over the discrete subset $\B \subset \t^*_+$. If $\E$ satisfies this condition, then it is a consequence of \eqref{eqn:antinonabelterm} that the right hand side of \eqref{eqn:antinonabel} is a locally finite sum.
\begin{definition}
Let $(M,\E,\L)$ be as above. Suppose the moment map $\mu_\g$ is proper and $\E$ satisfies \eqref{eqn:condition}. Then we define the $W$-anti-symmetric formal equivariant index $Q(M,\L,\E)\in R^{-\infty}(T)$ by the right hand side of \eqref{eqn:antinonabel}. Working backwards one then has also a formal $G$-equivariant index $Q_G(M,\L,\E)\in R^{-\infty}(G)$.
\end{definition} 

\subsection{Tensor powers $\L^k$ and piecewise quasi-polynomials.}
We continue to assume $\mu_\g$ is proper and $\E$ satisfies \eqref{eqn:condition}. Then for any $k \in \bZ_{>0}$, the formal anti-symmetric index $Q(M,\L^k,\E) \in R^{-\infty}(T)$ is defined. 
\begin{definition}
Let $m \colon \bZ_{>0}\times \Lambda \rightarrow \bZ$ be the function such that $m(k,-)$ is the anti-symmetric multiplicity function for $Q(M,\L^k,\E) \in R^{-\infty}(T)$.
\end{definition}

In Section \ref{sec:distfam} we define a space $\S(\Lambda)$ of `piecewise quasi-polynomial functions' on $\bZ_{>0}\times \Lambda$, involving locally finite sums of quasi-polynomials for the lattice $\bZ \oplus \Lambda$ multiplied with characteristic functions $[C_{P,\sigma}]$ where $P\subset \t^*$ is a rational polyhedron, $\sigma \in \Lambda \otimes \bQ$ and $C_{P,\sigma}$ is the `shifted cone'
\[ C_{P,\sigma}=\{(t,t\xi+\sigma)\in \bR \oplus \t^*|\xi \in P, t>0\}. \]
We refer the reader to Section \ref{sec:distfam} for details. 

\begin{proposition}
\label{prop:inS}
The multiplicity function $m \in \S(\Lambda)$.
\end{proposition}
\begin{proof}[Proof (in a slightly simplified setting).]
It suffices to consider the terms in \eqref{eqn:antinonabel} one at a time, so fix $\beta \in \B$. In order to keep the discussion brief, we will make a couple of simplifying assumptions about the geometry: (i) the normal bundle $N_\beta$ is trivial, (ii) $Z_\beta$ is connected and lies in the generic infinitesimal $T$ orbit stratum of the component of $M^\beta$ containing $Z_\beta$, say that associated to $\h\subset \t$. These are not serious restrictions. The case when $N_\beta$ is non-trivial is only a little more involved (cf. \cite{LMVerlindeSums, VergneFormalAhat}). And one can arrange for (ii) to hold by a suitable perturbation of the symbol near $Z_\beta$, leading to new $Z_{\ti{\beta}}$'s, and then treating the components of the latter one at a time.

Under assumption (ii), we have the connected subtorus $H \subset T$ with Lie algebra $\h$ that fixes a neighborhood $U_\beta$ of $Z_\beta$ in $M^\beta$, and such that the induced action of $T/H$ has finite stabilizers. For convenience fix an isomorphism $T\simeq H \times T/H$, which determines an isomorphism $\Lambda\simeq \Lambda_H\times \Lambda_{T/H}$ of weight lattices. Passing to $H$-isotypical subbundles, $\sigma^{\L^k}_{Z_\beta}$ decomposes into a finite direct sum
\begin{equation} 
\label{eqn:sigmaLk}
\sigma^{\L^k}_{Z_\beta}=\bigoplus_j \sigma_j^{\L^k} \otimes \bC_{k\bar{\beta}+\gamma_j},
\end{equation}
where $\sigma_j^{\L^k}$ is $T/H$-transversally elliptic, $\bar{\beta}$ is the image of $\beta \in \t^*$ under the quotient map to $\t^*/\tn{ann}(\h)=\h^*$ (the image is necessarily integral for $H$ since it is the weight of the $H$ action on $\L\upharpoonright Z_\beta$), and $\gamma_j$ are weights for the $H$ action (not depending on $k$). 

Under assumption (i), and making use of \eqref{eqn:sigmaLk}, \eqref{eqn:antinonabelterm} simplifies to a (sum of) products
\[ \sum_j \tn{index}_{T/H}(\sigma_j^{\L^k}) \otimes \big(\bC_{k\bar{\beta}+\gamma_j}\otimes \Sym(N^{1,0}_{\beta,z})\big) \in R^{-\infty}(T/H)\otimes R^{-\infty}(H),\]
where $N_{\beta,z}$ is a typical fibre of the normal bundle $N_\beta$. The corresponding multiplicity function is therefore a product 
\[ \sum_j m^{\L^k}_{\beta,j} \otimes \big(\delta_{k\bar{\beta}+\gamma_j}\star \P(\Phi_\beta)\big),\]
where $m^{\L^k}_{\beta,j}$ is the multiplicity function for $\tn{index}_{T/H}(\sigma_j) \in R^{-\infty}(T/H)$ and $\P(\Phi_\beta)$ is the Kostant partition function associated to the list of weights $\Phi_\beta$ for the $H$ action on the complex vector space $N^{1,0}_{\beta,z}$. The $k$-dependent translate of the partition function $\delta_{k\bar{\beta}+\gamma_j}\star \P(\Phi_\beta)$ belongs to the space $\S(\Lambda_H)$ of piecewise quasi-polynomial functions.

Finally we claim that $m^{\L^k}_{\beta,j}$ is quasi-polynomial (as a function of $(k,\lambda)\in \bZ \oplus \Lambda_{T/H}$). Let us further simplify to the case where $T/H$ acts freely on $U_\beta$, and let $F_\beta=U_\beta/T$, a smooth manifold. Then $\sigma_j$ gives rise to an elliptic operator $D$ on $F_\beta$, $\L$ to a line bundle on $F_\beta$, still denoted by $\L$, and each element $\lambda\in \Lambda_{T/H}$ gives rise to an associated line bundle $\L_\lambda=U_\beta\times_T \bC_{\lambda}$ on $F_\beta$. The multiplicity function for the index of the transversally elliptic symbol $\sigma_j$ twisted by $\L^k$ is given 
by $m^{\L^k}_{\beta,j}(\lambda,k)=\tn{index}(D^{\L^k\otimes \L_\lambda})\in \bZ$, the index of the elliptic operator $D$ twisted by $\L^k\otimes \L_\lambda$. The Atiyah-Singer formula for the index depends on $\lambda,k$ through the Chern character of $\L^k\otimes \L_\lambda$, so is a polynomial function of $(k,\lambda)$ on $\bZ\oplus \Lambda_{T/H}$. Similarly when the action of $T/H$ has finite stabilizers, we obtain a 
quasi-polynomial function of $(k,\lambda)\in \bZ \oplus \Lambda_{T/H}$, using the index formula for orbifolds.
\end{proof}
\ignore{Finally we claim that $m^{\L^k}_{\beta,j}$ is quasi-polynomial (as a function of $(k,\lambda)\in \bZ \oplus \Lambda_{T/H}$).  Let $P \rightarrow U_\beta$ be the principal $U(1)$-bundle associated to $\L|_{U_\beta}$. The pull-back of $\sigma^{\L}_j$ to the total space of $P$ is a $U(1)\times T/H$-transversally elliptic symbol $\sigma^P_j$, and the multiplicity function $m^P_j \colon \bZ \oplus \Lambda_{T/H}\rightarrow \bZ$ for its $U(1)\times T/H$-equivariant index is exactly the function $(k,\lambda)\mapsto m^{\L^k}_{\beta,j}(\lambda)$. We then use a fundamental property of indices of transversally elliptic symbols: the index is supported (as a distribution) in the union of the stabilizers of points in the manifold. Thus $\tn{index}_{U(1)\times T/H}(\sigma^P_j)$ is a finite sum of derivatives of delta distributions in $U(1)\times T/H$ at rational points hence, by Fourier transform, $m^P_j$ is quasi-polynomial.}

\subsection{Semi-classical asymptotics.}
Let $m \colon \bZ_{>0}\times \Lambda \rightarrow \bZ$ be a $W$-anti-symmetric multiplicity function describing a formal $G$-equivariant index, as in the previous section. Define a sequence of distributions on $\t^*$ (see Section \ref{sec:disttheta})
\[ \pair{\Theta(m;k)}{\varphi}=\sum_{\lambda \in \Lambda}m(k,\lambda)\varphi \big(\tfrac{\lambda}{k}\big).\]
The main result Theorem \ref{thm:AsymExp} of Section \ref{sec:distfam} shows that this sequence admits an asymptotic expansion in powers of $k$ (with coefficients that are periodic in $k$) as $k \rightarrow \infty$ denoted $\A(m;k)$.

In this geometric setting, $\A(m;k)$ has an interpretation as a formal series of twisted Duistermaat-Heckman distributions. The leading term of this series is a Duistermaat-Heckman distribution, and the lower terms bring in corrections from the $\Ahat$-form. To explain this we must introduce further notation. Recall we assumed $\mu_\g \colon M \rightarrow \g^*$ is proper. Let $\mu=\tn{pr}_{\t^*}\circ \mu_\g$ be the projection to $\t^*$, and let $\omega(X)=\omega+2\pi \i\pair{\mu}{X}$ be the $T$-equivariant 2-form. Note that $\mu$ is not proper in general.

Let $\gamma(X)$ be a closed $T$-equivariant differential form on $M$ depending smoothly on $X$ for $X$ in a neighborhood of $0 \in \t$, and such that the restriction of $\mu$ to the support of $\gamma(X)$ in $M$ is a proper map. Taking the Taylor expansion at $0 \in \t$, we replace $\gamma(X)$ with a formal power series in $X$ with coefficients in $\Omega(M)$. The terms of this series may be further rearranged into terms of homogeneous total degree (recall elements of $\Omega^p(M)$ have degree $p$, and $X$ has degree $2$). We let $\gamma_n(X)$ denote the homogeneous component of total degree $n$; it may be further decomposed into a finite sum
\[ \gamma_n(X)=\sum_{j=0}^{\dim(M)} \gamma_{nj} p_j(2\pi \i X) \]
where $\gamma_{nj} \in \Omega^j(M)$ and $p_j$ is a homogeneous polynomial of degree $(n-j)/2$. Since $\gamma_n(X)$ has polynomial dependence on $X$ and $\mu$ is proper on its support, we can define the $\gamma_n$-twisted Duistermaat-Heckman distribution $\tn{DH}(\gamma_n) \in \D'(\t^*)$ by
\begin{equation} 
\label{eqn:DH}
\pair{\tn{DH}(\gamma_n)}{\varphi}=\sum_{j=0}^{\dim(M)} \int_M e^\omega \gamma_{nj} \mu^*(p_j(\partial)\varphi).
\end{equation}
If $M$ is compact then this is equivalent to
\[ \pair{\tn{DH}(\gamma_n)}{\varphi}=\int_{M\times \g} e^{\omega(X)}\gamma_n(X)\hat{\varphi}(X) \]
where $\hat{\varphi}$ is the Fourier transform (a smooth rapidly decreasing density on $\g$).  We define the formal series of distributions
\begin{equation} 
\label{eqn:formalDH}
\DH(\gamma;k)=k^{\dim(M)/2}\sum_{n=0}^\infty \frac{1}{k^n}\DH(\gamma_n).
\end{equation}
If $\gamma(X)$ is polynomial in $X$ to begin with, then \eqref{eqn:formalDH} is a finite sum and is equivalent to taking the $\gamma$-twisted Duistermaat-Heckman distribution with respect to the equivariant 2-form $k\cdot \omega(X)$ and then performing the re-scaling $X \leadsto X/k$.

For the application to formal indices, we make the choice (topologist's normalization):
\begin{equation} 
\label{eqn:gamma}
\gamma(X)=\Ahat(M,X)\Ch(\E/\sf{S},X)\Ch(B,X),
\end{equation}
where $\Ch(\E/\sf{S},X)$ is the equivariant twisting Chern character \cite{BerlineGetzlerVergne}, and $\Ch(B,X)$ is the pullback by $\tn{pr}_{(\t^*)^\perp}\circ \mu_\g$ of the equivariant Chern character of the Bott-Thom element for $(\t^*)^\perp\simeq \t^\perp$. A formula for $\Ch(B,X)$ may be given in terms of a compactly supported equivariant Thom form $\tau_{\t^\perp}(X)$ on $(\t^\perp)^*$:
\[(\tn{pr}_{(\t^*)^\perp}\circ \mu_\g)^*\tau_{\t^\perp}(X)\tn{det}_{\t^\perp_-}\Big(\frac{1-e^{\ad_X}}{\ad_X}\Big).\]
The pullback of $\tau_{\t^\perp}(X)\tn{det}_{\t^\perp_-}(1-e^{\ad_X}/\ad_X)$ to $0 \in \t^\perp$ is
\[ \tn{det}_{\t^\perp_-}(1-e^{\ad_X})=\prod_{\alpha \in \mf{R}_-}1-e^{2\pi \i\pair{\alpha}{X}}.\]

Since $\mu_\g$ is proper, $\mu$ is proper on the support of $\Ch(B,X)$, and hence on the support of $\gamma(X)$. When $\mu$ is itself proper, $\Ch(B,X)$ can be replaced with $\tn{det}_{\t^\perp_-}(1-e^{\ad_X})$.
In any case we have the following.
\begin{theorem}[compare Theorem 3.12 in \cite{ParVerSemiclassical}]
\label{thm:asymDH}
With $\gamma$ as in \eqref{eqn:gamma}, one has
\begin{equation} 
\label{eqn:ADH}
\A(m;k)=\DH(\gamma;k). 
\end{equation}
\end{theorem}
Intuitively Theorem \ref{thm:asymDH} says that the Berline-Vergne equivariant cohomology formula for the index of a twisted Dirac operator on a \emph{compact} manifold $M$:
\[\tn{index}_T(\dirac^{\L^k}\wh{\otimes}\!\wedge \!\t^\perp_-)(\exp X/k)=\int_M e^{k\omega(X/k)} \Ahat(M,X/k)\Ch(\E/\sf{S},X/k)\tn{det}_{\t^\perp_-}(1-e^{\ad_{X/k}})\]
for $X$ in the Lie algebra of $T$, is valid in the asymptotic sense
when $M$ is non-compact and $k$ tends to $\infty$, provided we first replace $\tn{det}_{\t^\perp_-}(1-e^{\ad_{X/k}})$ with $\Ch(B,X/k)$ (to handle $\mu$ being not necessarily proper), and then replace $\gamma(X/k)=\Ahat(M,X/k)\Ch(\E/\sf{S},X/k)\Ch(B,X/k)$ by its graded series of equivariant polynomial forms.

We refer the reader to \cite[Section 4]{ParVerSemiclassical} for proof of a similar result (the method applies here as well). The proof is based on applying an index formula for transversally elliptic symbols due to Paradan, Berline and Vergne to each term \eqref{eqn:antinonabel} in order to compute $\A(m;k)$ in terms of twisted Duistermaat-Heckman distributions localized near each $Z_\beta$. Then the non-abelian localization formula in equivariant \emph{cohomology} \cite{Paradan98} is used (in reverse) to sum up the contributions from all $\beta \in \B$, leading to the right hand side of \eqref{eqn:ADH}.

\subsection{Functoriality under restriction.}
Let $G'\subset G$ be a compact connected Lie subgroup. Assume that the restriction of the $G$-action to $G'$ still satisfies the assumptions guaranteeing that the formal $G'$-equivariant index $Q_{G'}(M,\L,\E)$ is well-defined, namely: (i) the moment map $\mu_{\g'}=\tn{pr}_{(\g')^*}\circ \mu_\g$ for the $G'$ action is proper; (ii) the $G'$ analogue of \eqref{eqn:condition}, expressed in terms of the vanishing locus $Z_{G'}$ of the new Kirwan vector field $\kappa'$ and so on.  It then makes sense to ask whether the restriction of the formal $G$-equivariant index $Q_G(M,\L,\E)$ to $G'$ coincides with $Q_{G'}(M,\L,\E)$. As mentioned already, in case $M$ is compact, the complicated definition of the formal equivariant index can be replaced with the naive one in terms of the equivariant index of the Dirac operator, and it is then obvious that the restriction of $Q_G(M,\L,\E)$ to $G'$ equals $Q_{G'}(M,\L,\E)$.  When $M$ is non-compact this statement is still true, but becomes highly non-trivial (see \cite{ParadanFormalI,ParadanFormalII,ParadanFormalIII}). We outline here an approach to the proof of this statement which relies crucially on one of the main theorems of this article (Theorem \ref{thm:unicity}).

The corresponding objects for $G'$ will be denoted using $^\prime$, for example $T'$ is a maximal torus for $G'$, $\mu' \colon M \rightarrow (\t')^*$ is the moment map for the $T'$ action, $\Lambda' \subset (\t')^*$ is the (real) weight lattice for $T'$, $m' \colon \bZ_{>0}\times \Lambda'\rightarrow \bZ$ is the $W'$-anti-symmetric multiplicity function, and so on. Choose maximal tori such that $T'\subset T$, and let $\pi \colon \t^* \rightarrow (\t')^*$ be the canonical map. Select positive roots $\mf{R}_+$, $\mf{R}'_+$ as follows: let $X' \in \t' \subset \t$ be a polarizing vector for $\mf{R}'$, which determines a set of positive roots $\mf{R}'_+$. If necessary perturb $X'$ in $\t$ to obtain a polarizing vector $X \in \t$ for $\mf{R}$, determining a set of positive roots $\mf{R}_+$. Assuming the perturbation is sufficiently small then $\mf{R}_+' \subset \pi(\mf{R}_+)$. Then $(\t')^\perp_-$ (orthogonal complement taken in $\g'$) is a complex $T'$-invariant subspace of $\t^\perp_-$, hence the quotient $\t^\perp_-/(\t')^\perp_-$ is a complex $T'$-space. Let $\Phi^{\g,\g'}_-$ denote the list of weights for the $T'$ action on $\t^\perp_-/(\t')^\perp_-$; as a list (with multiplicities) $\Phi^{\g,\g'}_-=\pi(\mf{R}_-)-\mf{R}'_-$.

The restriction of $Q_G(M,\L^k,\E)$ to $G'$ has a $W'$-anti-symmetric multiplicity function that we denote by $\tn{res}^G_{G'} m$ (a function of $(k,\lambda') \in \bZ_{>0}\times \Lambda'$). The functoriality-under-restrictions property for formal equivariant indices is equivalent to the following statement:
\begin{theorem}
\label{thm:functoriality}
$\tn{res}^G_{G'}m=m'$.
\end{theorem}
By Proposition \ref{prop:inS} and Theorem \ref{thm:unicity}, the theorem follows if we can show $\A(g'\cdot \tn{res}^G_{G'}m)=\A(g'\cdot m')$ for each $g' \in T'_\bQ$. We will explain the argument for the case $g'=1$. The general case is similar, but involves more complicated twisted Duistermaat-Heckman distributions that appear in the asymptotic expansion of the formula for indices of transversally elliptic symbols near $g'\ne 1$; see \cite[Section 4, 5]{ParVerSemiclassical} for details.

\begin{proof}[Proof that $\A(\tn{res}^G_{G'}m)=\A(m')$.]
Let 
\[ \P(\Phi^{\g,\g'}_-)\colon \Lambda'\rightarrow \bZ \]
be the Kostant partition function for the list of weights $\Phi^{\g,\g'}_-$ for the $T'$ action on $\t^\perp_-/(\t')^\perp_-$. If $m$ is compactly supported, then it follows from the Weyl character formula that
\[ \tn{res}^G_{G'}m=\pi_\ast m \star \P(\Phi^{\g,\g'}_-).\]
More generally a version of this formula holds if we cut off $m$ and take the limit as the cut off is removed. Let $K \subset \t^*$ be a compact rational polytope, invariant under the (unshifted) action of $W$, and containing an open neighborhood of the origin (if $G$ is semisimple, the convex hull of the $W$-orbit of $\rho$ works). For $\ell \in \bZ_{>0}$, let $\ell\cdot K$ denote its dilation by factor $\ell$. Let $m_\ell=m\cdot [C_{\ell\cdot K,-\rho}]$.  Then $m_\ell(k,-)$ is compactly supported and anti-symmetric for the $\rho$-shifted $W$-action, for each $k$, $\ell$. And we have
\begin{equation}
\label{eqn:res}
\tn{res}^G_{G'}m=\pi_\ast m \Star_{\tn{lim}} \P(\Phi^{\g,\g'}_-):=\lim_{\ell \rightarrow \infty}\pi_\ast m_\ell \star \P(\Phi^{\g,\g'}_-).
\end{equation}
By Proposition \ref{prop:pushm} (recall convolution is push-forward under addition),
\begin{equation} 
\label{eqn:pushRm}
\A(\tn{res}^G_{G'}m)=\pi_\ast \A(m)\Star_{\tn{lim}}\A(\P(\Phi^{\g,\g'}_-)).
\end{equation}
By Theorem \ref{thm:asymDH}, $\A(m;k)=\DH(\gamma;k)$, with $\gamma$ the equivariant form in \eqref{eqn:gamma}. Since we assume $\mu_{\g'}$ is proper, in the formula for $\gamma(X)$ we may replace $\Ch(B,X)$ with the product of the Chern character of the Bott element for $(\t'_-)^\perp$ and $\tn{det}_{\t^\perp_-/(\t')^\perp_-}(1-e^{\ad_X})$. For $X \in \t'$, $\tn{det}_{\t^\perp_-/(\t')^\perp_-}(1-e^{\ad_X})=\prod_{\alpha \in \Phi^{\g,\g'}_-}(1-e^{2\pi \i\pair{\alpha}{X}})$. Functoriality under restriction is immediate for Duistermaat-Heckman distributions, and gives the equation
\begin{equation}
\label{eqn:DHfunc}
\pi_\ast \DH(\gamma;k)=\prod_{\alpha \in \Phi^{\g,\g'}_-}\big(1-e^{-\partial_\alpha/k}\big)\DH'(\gamma';k).
\end{equation}
Substituting \eqref{eqn:DHfunc} into \eqref{eqn:pushRm} and using $\A(m;k)=\DH(\gamma;k)$, $\A(m';k)=\DH'(\gamma';k)$ yields
\[ \A(\tn{res}^G_{G'}m;k)=\prod_{\alpha \in \Phi^{\g,\g'}_-}\big(1-e^{-\partial_\alpha/k}\big)\A(m';k)\Star_{\tn{lim}}\A(\P(\Phi^{\g,\g'}_-);k).\]
The derivatives can be moved to the other factor in the convolution, and then moved inside $\A(\cdot)$ using the property (Proposition \ref{prop:module}) $(1-e^{-\partial_\eta/k})\A(-;k)=\A(\nabla_{\eta}-;k)$, where $\nabla_{\eta} f(\lambda)=f(\lambda)-f(\lambda-\eta)$ denotes the finite difference operator. Since
\[ \Bigg(\prod_{\alpha \in \Phi^{\g,\g'}_-}\nabla_{\alpha} \Bigg)\P(\Phi^{\g,\g'}_-)=\delta_0 \]
is the delta function supported at $0 \in \Lambda'$, we obtain the desired equality
\[ \A(\tn{res}^G_{G'}m)=\A(m')\Star_{\tn{lim}}\A(\delta_0)=\A(m').\]
\end{proof}

\section{A family of distributions associated to a quasi-polynomial function}\label{sec:distfam}
In this section we define a certain space $\S(\Lambda)$ of `piecewise quasi-polynomial functions' on a lattice $\bZ \oplus \Lambda$.  For each $m \in \S(\Lambda)$ we introduce a family of distributions $\Theta(m;k)$, $k \in \bZ_{>0}$, and prove that they admit asymptotic expansions as $k \rightarrow \infty$.

\subsection{Quasi-polynomial functions.}
Let $W$ be a finite dimensional real vector space, and let $\Gamma$ be a full-rank lattice in $W$.  A function $m \colon \Gamma \rightarrow \bC$ is \emph{periodic} if there is a non-zero integer $D$ such that $m(\gamma_0+D\gamma)=m(\gamma_0)$ for all $\gamma_0,\gamma \in \Gamma$.  The space of such functions is spanned by $\gamma \mapsto g^\gamma$ with $g \in \Gamma^\wedge_\bQ$.  By definition, the algebra $\QPol(\Gamma)$ of \emph{quasi-polynomial} functions on $\Gamma$ is generated by polynomials and periodic functions.  It is $\bZ$-graded by polynomial degree, i.e. $\QPol_n(\Gamma)$ is spanned by functions $\gamma \mapsto g^\gamma h(\gamma)$ where $g \in \Gamma^\wedge_{\bQ}$ and $h \in \Pol_n(W)$ is a homogeneous polynomial of degree $n$ on $W$.  

For any $q \in \QPol(\Gamma)$ there is a sublattice $\Gamma^\prime \subset \Gamma$ of finite index such that $q$ is polynomial on each coset of $\Gamma^\prime$, that is, for each coset $[\mu] \in \Gamma/\Gamma^\prime$ there is a polynomial $\ti{q}_{[\mu]}$ such that $\ti{q}_{[\mu]}(\gamma)=q(\gamma)$ for $\gamma \in [\mu]=\mu+\Gamma^\prime$.  There is a unique decomposition
\begin{equation} 
\label{eqn:Defqg}
q(\gamma)=\sum_{g \in \Gamma^\wedge_\bQ}g^{-\gamma} q_g(\gamma)
\end{equation}
where $q_g$ is polynomial, and is identically $0$ unless $g$ is in the finite subgroup $(\Gamma^\prime)^\ast/\Gamma^\ast \subset \Gamma^\wedge_{\bQ}$.  The Fourier transform $\hat{q}$ is a distribution on $\Gamma^\wedge$ supported at the finite set of points $g$ such that $q_g \ne 0$ (this also makes uniqueness of \eqref{eqn:Defqg} clear).  By finite Fourier transform, for $g \in (\Gamma^\prime)^\ast/\Gamma^\ast$ one has
\begin{equation}
\label{eqn:FiniteFourier}
q_g(\gamma)=\frac{1}{\#(\Gamma/\Gamma^\prime)}\sum_{[\mu]\in \Gamma/\Gamma^\prime} g^\mu \ti{q}_{[\mu]}(\gamma).
\end{equation}

\subsection{Piecewise quasi-polynomial functions.}
At first glance one might want to say (imitating piecewise polynomial functions) that a function $f \colon \Gamma \rightarrow \bC$ is `piecewise quasi-polynomial' if $W$ can be expressed as a locally finite union of polyhedra $\{Q\}_{Q \in \Q}$, such that for each $Q$ the restriction of $f$ to $Q\cap \Gamma$ equals the restriction of a quasi-polynomial function.  Under this definition any function is piecewise quasi-polynomial, because one can always choose the $Q$ such that $Q \cap \Gamma$ is finite.

To obtain a substantive definition, and also motivated by applications, we will strongly restrict the type of locally finite decompositions into polyhedra that are allowed.  We first require that $W$ (resp. $\Gamma$) be of the form $\bR \oplus V$ (resp. $\bZ \oplus \Lambda$), where $V$ is a finite dimensional real vector space and $\Lambda \subset V$ is a full rank lattice.
\begin{remark}
\label{rem:QPolFamily}
Any element of $\QPol(\bZ\oplus \Lambda)$ can be written as a linear combination of functions of the form $j(k)k^d q(\lambda)$ where $j(k)$ is a periodic function of $k$, $d \in \bZ_{\ge 0}$, and $q$ is a quasi-polynomial function on $\Lambda$.  It is sometimes useful to take the point of view that $q \in \QPol(\bZ \oplus \Lambda)$ is a family of quasi-polynomials on $\Lambda$, depending in a quasi-polynomial way on a parameter $k \in \bZ$.
\end{remark}
Any subset $R \subset V$ generates a cone
\[ C_R=\{(t,tv) \in \bR \oplus V|v \in R, t>0 \} \subset \bR \oplus V.\]
In particular a rational polyhedron $P \subset V$ generates a cone $C_P$.  For $\sigma \in \Lambda\otimes \bQ$ we also define a translated cone
\[ C_{P,\sigma}=C_P+\sigma=\{(t,tv+\sigma)\in \bR \oplus V|v \in P, t>0\} \subset \bR \oplus V. \]
Let $[C_{P,\sigma}]$ denote the characteristic function of $C_{P,\sigma}$.  The only polyhedra permitted in the decomposition of our `piecewise quasi-polynomial functions' will be translated cones $Q=C_{P,\sigma}$.
\begin{remark}
\label{rem:uniqueopencone}
A basic fact is that if $P$ has non-empty interior and $f \colon \bZ \oplus \Lambda \rightarrow \bC$ is any function, then there is \emph{at most one} quasi-polynomial function that agrees with $f$ on $C_{P,\sigma} \cap (\bZ \oplus \Lambda)$.  
\end{remark}
\begin{remark}
\label{rem:crosssec}
It is sometimes useful to think of $C_{P,\sigma}$ in terms of the sequence of `cross-sections' $C_{P,\sigma}\cap (\{k\}\times V)=\{k\}\times (kP+\sigma)=\{k\} \times k(P+\sigma/k)$.  As a small example, suppose $\sigma_1, \sigma_2 \in \tn{lin}(P)$.  Let $R$ be any compact subset of the relative interior of $P$.  Then there exists a $K>0$ (depending on $R$, $\sigma_1$, $\sigma_2$) such that $R \subset (P+\sigma_1/k)\cap (P+\sigma_2/k)$ for $k>K$, hence $kR \subset (kP+\sigma_1)\cap (kP+\sigma_2)$, and it follows that $C_{P,\sigma_1}\cap C_{P,\sigma_2}$ contains $C_R \cap \{k>K\}$.  One has a similar statement if $\sigma_1,\sigma_2$ project to the same element of $V/\tn{lin}(P)$.
\end{remark}
In the next definition and in the rest of the article, we will usually not distinguish in notation between $[C_{P,\sigma}]$ and the restriction of $[C_{P,\sigma}]$ to the lattice $\bZ \oplus \Lambda$.
\begin{definition}
\label{def:slambda}
Define $\S(\Lambda)$ to be the space of functions on $\bZ\oplus \Lambda$ of the form
\begin{equation} 
\label{eqn:locfindecomp}
\sum_{P \in \P} \sum_{\sigma \in \Sigma_P} q_{P,\sigma}[C_{P,\sigma}]
\end{equation}
where $q_{P,\sigma}$ is a quasi-polynomial function on $\bZ \oplus \Lambda$, and the sum ranges over a collection $\P$ of convex polyhedra and shift vectors $\Sigma_P \subset \Lambda \otimes \bQ$ such that the collection of polyhedra 
\[\{P+[0,1]\sigma|P \in \P, \sigma \in \Sigma_P\}\] 
is locally finite, where $P+[0,1]\sigma=\{v+r\sigma|v \in P,r\in [0,1]\}$.
\end{definition}
The space $\S(\Lambda)$ is invariant under translations by elements of $\Lambda$ and is a $\QPol(\bZ \oplus \Lambda)$-module.  For $h \in \QPol(\bZ\oplus \Lambda)$ and $m \in \S(\Lambda)$ we let
\[ (h\cdot m)(k,\lambda)=h(k,\lambda)m(k,\lambda).\]
In particular for $g \in \Lambda^\vee_\bQ$, the function $\lambda \mapsto g^\lambda$ is quasi-polynomial, and we write $g\cdot m$ for the corresponding function:
\[ (g\cdot m)(k,\lambda)=g^\lambda m(k,\lambda).\]
\begin{remark}
\label{rem:locfincone}
The local finiteness condition implies that the collection $\P$ is itself locally finite and that for each $P \in \P$, $\Sigma_P$ is finite (but it is stronger than these two conditions).  Let $B \subset V$ be a bounded open set.  Note that $B \cap (P+[0,1]\sigma) \ne \emptyset$ if and only if there is a $r \in (0,1]$ such that $B \cap (P+r\sigma) \ne \emptyset$, if and only if there is a $r^{-1}=t \in [1,\infty)$ such that $tB \cap (tP+\sigma) \ne \emptyset$.  Thus an equivalent local finiteness condition is that for every bounded set $B$, the intersection $C_B \cap C_{P,\sigma} \cap (\bR_{\ge 1} \times V)=\emptyset$ for all but finitely many pairs $(P,\sigma)$.
\end{remark}
\begin{remark}
The representation of $m$ as a sum in \eqref{eqn:locfindecomp} is not unique.  For example, consider $V=P=\bR \supset \bZ=\Lambda$, $P_+=\bR_{\ge 0}$, $P_-=\bR_{\le 0}$ and $P_0=\{0\}$.  Then $[C_P]=[C_{P_+}]+[C_{P_-}]-[C_{P_0}]$.  The choice of shifts $\sigma$ is also not unique, that is, different choices of $\sigma$ can give the same function when restricted to the lattice $\bZ \oplus \Lambda$.  For example $[C_{P_+,\sigma}\cap (\bZ \oplus \Lambda)]$ does not change as $\sigma$ varies in the interval $(-1,0] \cap \bQ$.
\end{remark}

\begin{example}
\label{ex:inclusionexclusion}
An important example of a function $m \in \S(\Lambda)$ is the following.  Suppose $P$ is a rational polyhedron, and that we have a finite covering $P=\cup_\alpha P_\alpha$ by closed rational polyhedra $P_\alpha$.  Let $m$ be a function on $C_P \cap (\bZ \oplus \Lambda)$ such that its restriction to each $C_{P_\alpha}\cap (\bZ \oplus \Lambda)$ is given by a quasi-polynomial function.  Using inclusion-exclusion formulas, we may write $[P]$ as a signed sum of the characteristic functions $[P_\alpha]$, together with intersections $[P_\alpha \cap P_\beta]$, $[P_\alpha \cap P_\beta \cap P_\gamma]$, and so on.  Thus $m \in \S(\Lambda)$.
\end{example}
\begin{example}
\label{ex:growingshifts}
For $n \in \bZ_{> 0}=\{1,2,3,...\}$ define 1-dimensional polyhedra $P_n$ and shift vectors $\sigma_n$ by
\[ P_n=[-1,1]\times \{n\}, \qquad \sigma_n=(-n,-n).\]
For any collection of quasi-polynomials $q_n(k,\lambda)$ the function
\[ m=\sum_{n=1}^\infty q_n[C_{P_n}] \]
is in $\S(\Lambda)$.  Indeed the collection of parallelograms
\[ P_n+[0,1]\sigma_n, \qquad n \in \bZ_{> 0} \]
is locally finite.  By contrast, for the same set of polyhedra $P_n$, we cannot choose $\sigma_n=(0,-n)$, as then the family $P_n+[0,1]\sigma_n$, $n \in \bZ_{>0}$ is not locally finite.
\end{example}

\begin{remark}
\label{rem:changelattice}
It is occasionally useful to pass to a finer lattice $\ti{\Lambda}\supset \Lambda$.  There is a restriction map $\S(\ti{\Lambda})\rightarrow \S(\Lambda)$.  There is also a canonical splitting $\S(\Lambda)\hookrightarrow \S(\ti{\Lambda})$ given by extending $m(k,\lambda)$ by zero; to see that the result is in $\S(\ti{\Lambda})$, note that one can choose a decomposition of $m$ as in \eqref{eqn:locfindecomp}, and then extend each $q_P\colon \bZ \oplus \Lambda \rightarrow \bC$ by zero.
\end{remark}

\subsection{The distributions $\Theta(m;k)$.}\label{sec:disttheta}
The main object of study in this article is a family of distributions $\Theta(m;k)$ associated to a piecewise quasi-polynomial function $m$, that we introduce in this section.  For motivation, let $P$ be a rational convex polyhedron.  It is natural to consider the family of measures, for $k \in \bZ_{>0}$, defined by
\begin{equation} 
\label{eqn:ex0}
\varphi \mapsto \sum_{\lambda/k \in P,\, \lambda \in \Lambda} \varphi\big(\tfrac{\lambda}{k}\big),
\end{equation}
for any test function $\varphi$.  The support of this measure is a finite set of points contained in $P$, and fills $P$ more densely as $k \rightarrow \infty$ (see Figure \ref{fig:intropic} in the introduction). 

More generally let $\sigma \in \Lambda_\bQ$, and set $P_{\sigma,k}=P+\tfrac{\sigma}{k}$ for $k \in \bZ_{>0}$.  This is a family of polyhedra in $V$, which approaches $P$ as $k \rightarrow \infty$, and it is natural to consider the family of measures analogous to \eqref{eqn:ex0}, defined by
\begin{equation} 
\label{eqn:ex1}
\varphi \mapsto \sum_{\lambda/k \in P_{\sigma,k},\, \lambda \in \Lambda} \varphi\big(\tfrac{\lambda}{k}\big), 
\end{equation}
for any test function $\varphi$.  The support of this measure is contained in a neighborhood of $P$ of size $O(k^{-1})$, see Figure \ref{fig:shifted}.

Still more generally, we can consider weighted sums of measures as in \eqref{eqn:ex1}, with quasi-polynomial weights, and we are lead to the following:
\begin{definition}
\label{def:Theta}
Let $m \in \S(\Lambda)$.  For $k \in \bZ_{>0}$ let $\Theta(m;k)$ be the measure defined by
\[ \pair{\Theta(m;k)}{\varphi}=\sum_{\lambda \in \Lambda}m(k,\lambda)\varphi \big(\tfrac{\lambda}{k}\big),\]
for every test function $\varphi \in C_c^\infty(V)$.  Equivalently, $\Theta(m;k)$ is given by the equation
\[ \Theta(m;k)=\sum_{\lambda \in \Lambda}m(k,\lambda)\delta_{\lambda/k}\]
where $\delta_{\lambda/k}$ is the Dirac delta distribution at $\lambda/k \in V$.
\end{definition}
\begin{remark}
If $\ti{\Lambda} \supset \Lambda$ is a finer lattice, and $\ti{m}\in \S(\ti{\Lambda})$ denotes the extension-by-zero of $m \in \S(\Lambda)$ as in Remark \ref{rem:changelattice}, then $\Theta(m;k)=\Theta(\ti{m};k)$.
\end{remark}

\begin{figure}
{\centering
\includegraphics[clip, trim=3.4cm 18.5cm 2.3cm 4cm, width=1\textwidth]{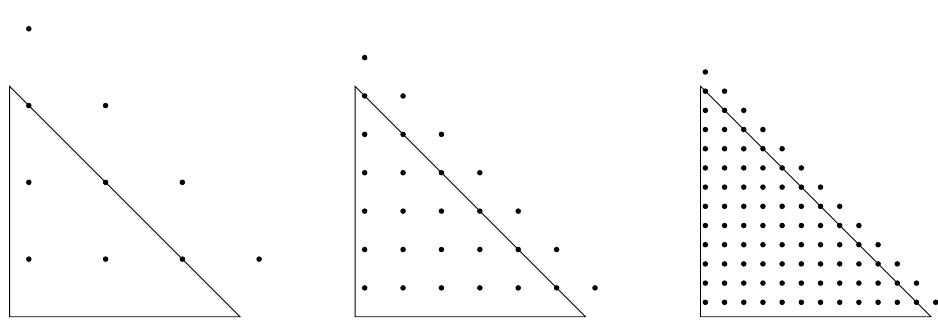}
}
\caption{Support of a measure for $P_{\sigma,k}$, $\Lambda=\bZ^2 \subset \bR^2=V$, $\sigma=\big(\tfrac{3}{2},\tfrac{9}{2}\big)$, $P=\{x\ge 0,y\ge 0, x+y \le 1\}$, and $k=3,6,12$.}
\label{fig:shifted}
\end{figure}

\subsection{Asymptotic expansion in $k$.}
Let $\Theta(k)$, $k \in \bZ_{>0}$ be a sequence of distributions on a vector space $V$.  Suppose that there are integers $D>0$ and $s$, and a sequence of distributions $\theta_n(k)$, $n \in \bZ_{\ge 0}$ such that $\theta_n(k+D)=\theta_n(k)$ and moreover for any test function $\varphi \in C^\infty_c(V)$ and any $N \in \bZ$ we have
\begin{equation} 
\label{eqn:DefAsymExp}
\pair{\Theta(k)}{\varphi}=k^s\sum_{n=0}^N\frac{1}{k^n}\pair{\theta_n(k)}{\varphi}+o(k^{s-N}).
\end{equation}
In this case we will say that $\Theta(k)$ \emph{admits an asymptotic expansion with coefficients of period $D$, and leading term at most $O(k^s)$}, and will write
\begin{equation}
\label{eqn:AsymNotation} 
\Theta(k) \sim k^s\sum_{n=0}^\infty \frac{1}{k^n}\theta_n(k).
\end{equation}
If an asymptotic expansion exists, then the $\theta_n(k)$ are uniquely determined.  We will use the notation $\A(k)$ to denote an asymptotic series, as on the right hand side of \eqref{eqn:AsymNotation}.  More generally, we will say that the sequence $\Theta(k)$ \emph{admits an asymptotic expansion} if there exists a partition of unity $\{\rho_j\}_{j \in J}$ and integers $D_j>0, s_j$ such that for each $j \in J$ the distribution $\rho_j \Theta(k)$ admits an asymptotic expansion with coefficients of period $D_j$, and leading term at most $O(k^{s_j})$ (we allow that the integers $D_j$, $s_j$ could grow without bound as one goes to infinity in $V$).  In this case if $\A_j(k)$ is the asymptotic expansion of $\rho_j\Theta(k)$, and $\A(k):=\Sigma_j \, \A_j(k)$, then we will write $\Theta(k)\sim \A(k)$.

If the asymptotic expansion \eqref{eqn:DefAsymExp} holds for all Schwartz functions $\varphi$, then it is valid to take the Fourier transform of both sides and obtain an asymptotic expansion $\wh{\A}(k)$ for the sequence of generalized functions $\wh{\Theta}(k)$ obtained by Fourier transform.  Here $\wh{\A}(k)$ is computed by taking the Fourier transform term-by-term in the expansion.

We will prove below (Theorem \ref{thm:AsymExp}) that the measures $\Theta(m;k)$, $m \in \S(\Lambda)$ admit an asymptotic expansion.  The necessity of allowing \emph{periodic} coefficients, rather than just constant coefficients, can be seen immediately in simple examples.  Consider for example $V=\bR \supset \bZ=\Lambda$, $P=[-1/2,1/2]$ and $m=[C_P]$.  One obtains an asymptotic expansion for $\Theta(m;k)$ using the Euler-Maclaurin formula.  Note that the intersection $P \cap k^{-1}\cdot \bZ$ contains the endpoints of the interval when $k$ is even, whereas for $k$ odd it does not contain the endpoints.  Because of the alternating boundary behavior, already the $O(1)$ term in the expansion will have period $2$.

For $\sigma \in V$ let $\partial_\sigma$ denote the directional derivative in the direction $\sigma$.  Let $e^{-\partial_\sigma/k}$ denote the formal series of differential operators obtained by substituting $\partial_\sigma/k$ into the Taylor expansion for $e^{-x}$ at $x=0$.  Given a formal series of distributions $\A(k)$ as in \eqref{eqn:AsymNotation} (with the powers of $k$ bounded above), we obtain a new formal series of distributions $e^{-\partial_\sigma/k}\A(k)$ by applying the differential operators to the summands in the obvious way.

The next result describes the behavior of asymptotic series for $\Theta(m;k)$ under the $\Pol(\bZ\oplus \Lambda)$-module action and under  translations by elements of the lattice.
\begin{proposition}
\label{prop:module}
Let $m \in \S(\Lambda)$ and suppose $\Theta(m;k)\sim \A(m;k)$.
\begin{enumerate}
\item Let $h=h(k,\lambda)$ be polynomial in $\lambda$ and quasi-polynomial in $k$.  Then
\[ \Theta(h\cdot m;k,v)=h(k,kv)\Theta(m;k,v) \sim h(k,kv)\A(m;k,v).\]
\item Let $\sigma \in \Lambda$.  Then the translation $\tau_\sigma m \in \S(\Lambda)$ and \[ \Theta(\tau_\sigma m;k)=\tau_{\sigma/k}\Theta(m;k)\sim e^{-\partial_\sigma/k}\A(m;k).\]
\end{enumerate}
\end{proposition}
\begin{proof}
Part (a) is immediate from the definition, using $h(k,\lambda)\delta_{\lambda/k}(v)=h(k,kv)\delta_{\lambda/k}(v)$.

Part (b) follows from the definition of $\Theta(m;k)$ and the following more general property of asymptotic expansions: let $\Theta(k)$ be a sequence of distributions on a vector space $V$ admitting an asymptotic expansion $\Theta(k)\sim \A(k)$, then $\tau_{\sigma/k}\Theta(k) \sim e^{-\partial_\sigma/k}\A(k)$.  To see this let $\varphi \in C_c^\infty(V)$ be a test function.  Then $\pair{\tau_{\sigma/k}\Theta(k)}{\varphi}=\pair{\Theta(k)}{\tau_{-\sigma/k}\varphi}$.  By Taylor's formula
\[ \tau_{-\sigma/k} \varphi \sim e^{\partial_\sigma/k}\varphi \]
with the asymptotic expansion being in the space $C^\infty_c(V)=\mathcal{D}(V)$ with its usual topology as a space of test functions (cf. \cite[Chapter 6]{RudinFunc}).  The usual argument that one can multiply asymptotic expansions (applied to $\pair{\Theta(k)}{\tau_{-\sigma/k}\varphi}$) then goes through, using the following fact: let $u_j \in \mathcal{D}^\prime(V)$ be a sequence of distributions converging to $u \in \mathcal{D}^\prime(V)$, and let $\phi_j \in \mathcal{D}(V)$ be a sequence of test functions converging to $\phi \in \mathcal{D}(V)$, then $\pair{u_j}{\phi_j} \rightarrow \pair{u}{\phi}$; in other words, the bilinear pairing $\mathcal{D}^\prime(V)\times \mathcal{D}(V) \rightarrow \bC$ is jointly sequentially continuous.  This last fact is an application of the uniform boundedness principle for Fr{\'e}chet spaces; see for example \cite[Theorem 2.1.8]{Hormander1} or \cite[Theorem 2.17 and Theorem 6.5]{RudinFunc}.
\end{proof}
\begin{remark}
Translation of the series $\Theta(m;k)$ corresponds to a `shear' transformation of $m$, that is, if $m^\prime(k,\lambda):=m(k,\lambda-k\sigma)$ then
\[ \Theta(m^\prime;k)=\tau_\sigma \Theta(m;k).\]
In particular if $P$ is a polyhedron then
\[ \Theta([C_{P+\sigma}];k)=\tau_\sigma \Theta([C_P];k).\]
One has the corresponding formula for the asymptotic series as well.
\end{remark}

\begin{definition}
\label{def:measures}
The lattice $\Lambda \subset V$ determines a normalization of Lebesgue measure on $V$ such that the volume of $V/\Lambda$ is $1$.  If $W \subset V$ is an affine rational subspace, then $\Lambda \cap \tn{lin}(W)$ determines a choice of Lebesgue measure on $W$ and we write $\delta_W$ for the corresponding distribution supported on $W$.  More generally if $P$ is any rational polyhedron then we obtain a canonical distribution $\delta_P$ supported on $P$, obtained by restricting the canonical measure on $\tn{aff}(P)$.  If $W_1$, $W_2$ are complementary rational subspaces in $V$ and $\Lambda=(\Lambda \cap W_1)\oplus (\Lambda \cap W_2)$, then the measure on $V$ is a product of the measures on $W_1$, $W_2$.
\end{definition}

We now define an analogue of the space $\S(\Lambda)$ of piecewise quasi-polynomial functions.
\begin{definition}
\label{def:Rmod}
Let $\R \subset \mathcal{D}^\prime(V)\times \bZ_{>0}$ be the space of families $\psi$ of distributions $\psi(k), k \in \bZ_{>0}$ such that there exists a locally finite collection $\P$ of rational polyhedra and a decomposition
\begin{equation} 
\label{eqn:decomppsi}
\psi(k)=\sum_F r_F(k)\delta_F 
\end{equation}
where $F$ runs over all faces of all $P \in \P$, and $r_F(k)$ is a differential operator on $V$ with polynomial coefficients and which is periodic in $k$.  We write $\R(\P)$ for the subspace of $\R$ consisting of those elements admitting a decomposition as above for a fixed locally finite collection $\P$.  Note that $\R$ and $\R(\P)$ are modules for the ring of differential operators on $V$ with polynomial coefficients.  
\end{definition}
\begin{remark}
\label{rem:germs}
For $\theta \in \R(\P)$, the distributions $\theta(k)$ have the following property: if $U_1$, $U_2$ are two open subsets that meet the same subset of the set of all faces of all polyhedra in $\P$, then $\theta(k)|_{U_1}=0 \Rightarrow \theta(k)|_{U_2}=0$.  For example, consider the collection $\P=\{P\}$ consisting of a single rational polyhedral cone $P$.  The distributions in $\R(\{P\})$ have the property that they are uniquely determined by their `germ' on any open neighborhood of any point in the apex set of $P$.
\end{remark}

\begin{theorem}
\label{thm:AsymExp}
For any $m \in \S(\Lambda)$, the family of measures $\Theta(m;k)$ admits an asymptotic expansion $\A(m;k)$, and the distributions appearing in the asymptotic expansion lie in the space $\R(\P)$, where $\P$ is the set of polyhedra appearing in a decomposition of $m$ as in \eqref{eqn:locfindecomp}.
\end{theorem}
\noindent We claim that it suffices to prove the result for $m=q[C_{P,\sigma}]$.  Indeed let $\varphi \in C^\infty_c(V)$ with $\supp(\varphi)=B$.  Note that the function $(k,\lambda) \in \bZ_{\ge 1}\times \Lambda \mapsto \varphi(\tfrac{\lambda}{k})$ has support contained in $C_B \cap (\bZ_{\ge 1} \times V)$.  Given $m$ decomposed as in \eqref{eqn:locfindecomp}, Remark \ref{rem:locfincone} shows that there is only a finite set (not depending on $k$) of summands $q_{P,\sigma}[C_{P,\sigma}]$ in \eqref{eqn:locfindecomp} that contribute to the pairing $\pair{\Theta(m;k)}{\varphi}$.

The proof for $m=q[C_{P,\sigma}]$ (given in Section \ref{sec:GenCase})  will be based on decomposing $P$ into unimodular cones and then applying the 1-dimensional case to each factor of a product.  In the 1-dimensional case one obtains an explicit asymptotic expansion from the Euler-Maclaurin formula (cf. \cite[Theorem 9.2.2]{CohenNumber}) that we briefly review next.

\subsection{Euler-Maclaurin formula and the 1-dimensional case.}
For $n=0,1,2,...$ the series
\[ b_n(x)=\sum_{j \in \bZ_{\ne 0}} \frac{e^{2\pi \i jx}}{(2\pi \i j)^n}\]
converges as a generalized function of $x \in \bR$.  We use standard Lebesgue measure to identify distributions and generalized functions.  By the Poisson summation formula
\begin{equation} 
\label{eqn:b0}
b_0(x)=\delta_{\bZ}(x)-1,
\end{equation}
where $\delta_{\bZ}(x)$ is the delta distribution concentrated on $\bZ \subset \bR$.  It is clear that $b_n$ is periodic of period $1$, and moreover
\begin{equation} 
\label{eqn:derivb}
\frac{d}{dx}b_n= b_{n-1}. 
\end{equation}
One can show that for $n \ge 1$, $b_n \in L^1_{\tn{loc}}(\bR)$, and
\begin{equation}
\label{eqn:intb}
\int_0^1 b_n(x) dx=0.
\end{equation}
Conversely \eqref{eqn:b0}, \eqref{eqn:derivb}, \eqref{eqn:intb} determine the $b_n$, and lead to a quick proof that for $n \ge 1$ one has
\begin{equation} 
\label{eqn:berpol}
b_n(x)=-\frac{1}{n!}B_n(\{x \}) 
\end{equation}
where $\{x \} \in [0,1)$ denotes the fractional part of $x$ and $B_n(x)=B_{n,1}(x)$ is the $n^{\tn{th}}$ Bernoulli polynomial, defined by the generating function
\begin{equation} 
\label{eqn:genfcn1}
\frac{te^{xt}}{e^t-1}=\sum_{n=0}^\infty B_n(x)\frac{t^n}{n!}.
\end{equation}
So for example 
\[ b_1(x)=\tfrac{1}{2}-\{x\}, \qquad b_2(x)=-\tfrac{1}{12}+\tfrac{1}{2}\{x\}-\tfrac{1}{2}\{x\}^2.\]
It is convenient to define $b_n(x)$ (for $n \ge 1$) for all $x \in \bR$ so that \eqref{eqn:berpol} holds.

Let $\zeta=e^{2\pi \i r} \ne 1$, $r \in \bQ$, and define the generalized function
\begin{equation} 
\label{eqn:defzetaber}
b_{n,\zeta}(x)=\sum_{m \in \bZ-r} \frac{e^{2\pi \i mx}}{(2\pi \i m)^n}.
\end{equation}
Note that in this case (i) $0 \notin \bZ-r=\{n-r|n \in \bZ\}$, so that the sum needs no restriction, and (ii) $b_{n,\zeta}(x-1)=\zeta b_{n,\zeta}(x)$, hence in particular $b_{n,\zeta}$ is periodic with period equal to the order of $\zeta$.  Similar to before, the Poisson summation formula gives
\begin{equation} 
\label{eqn:bzeta0}
b_{0,\zeta}(x)=e^{-2\pi \i rx}\delta_{\bZ}(x)=\sum_{\ell \in \bZ}\zeta^{-\ell} \delta_\ell.
\end{equation}
For $n \ge 1$, $b_{n,\zeta} \in L^1_{\tn{loc}}(\bR)$ and we have
\begin{equation} 
\label{eqn:derivbzeta}
b_{n,\zeta}^\prime(x)=b_{n-1,\zeta}(x), \qquad \int_0^1 e^{2\pi \i rx} b_{n,\zeta}(x)dx=(-2\pi \i r)^{-n}.
\end{equation}
As before \eqref{eqn:bzeta0}, \eqref{eqn:derivbzeta}, together with the formula $b_{n,\zeta}(x-1)=\zeta b_{n,\zeta}(x)$, determine the $b_{n,\zeta}$.  Define functions $B_{n,\zeta}(x)$ using the generating function
\begin{equation} 
\label{eqn:genfcn2}
\frac{te^{xt}}{\zeta e^t-1}=\sum_{n=1}^\infty B_{n,\zeta}(x)\frac{t^n}{n!} 
\end{equation}
and for convenience set $B_{0,\zeta}=0$.  By checking the properties that characterize $b_{n,\zeta}$ listed above, one obtains a quick proof of the formula
\begin{equation} 
\label{eqn:berzetapol} 
b_{n,\zeta}(x)=-\frac{1}{n!}\zeta^{\lceil -x \rceil} B_{n,\zeta}(\{x\}), \qquad n \ge 1.
\end{equation}
(Note that $\lceil a \rceil:=a+\{-a\} \in \bZ$.)  It is convenient to define $b_{n,\zeta}(x)$ (for $n \ge 1$) for all $x \in \bR$ so that \eqref{eqn:berzetapol} holds.

\begin{lemma}
\label{lem:EulerMac}
\label{cor:1dcase}
Let $s,\sigma \in \bQ$ and let $L=s+\bR_{\ge 0} \subset \bR$. Let $\zeta \in U(1)_\bQ$ and let $m(k,j)=\zeta^j[C_{L,\sigma}](k,j)$.  Then
\begin{equation} 
\label{eqn:EulerMac1}
\Theta(m;k) \sim k\zeta^{\lceil ks+\sigma \rceil}\Big(B_{0,\zeta}[L]-\sum_{n=1}^\infty \frac{1}{k^n}\frac{B_{n,\zeta}(\{-ks-\sigma\}+\sigma)}{n!}(-1)^{n-1}\delta_s^{(n-1)}\Big),
\end{equation}
where $\delta_s^{(n-1)}$ denotes the $(n-1)^{\tn{st}}$ derivative of the delta distribution $\delta_s(x)=\delta(x-s)$.  Recall that $B_{0,\zeta}=0$ unless $\zeta=1$, in which case $B_{0,1}=B_0=1$; in particular if $\zeta \ne 1$ then the leading term is $O(1)$ instead of $O(k)$. If, in place of $L$, the polyhedron is the full real line then,
\[ \sum_{j \in \bZ}\delta_{j/k} \sim k [\bR], \qquad \qquad \sum_{j \in \bZ}\zeta^j \delta_{j/k}\sim 0 \quad (\text{if }\zeta \ne 1).\]
\end{lemma}
\begin{proof}
Let $f$ be a Schwartz function and $x_0,x_1 \in \bR \setminus \bZ$.  For $n \in \bZ_{\ge 0}$ define
\[ R_n(f)=\int_{x_0}^{x_1} f^{(n)}(x)b_n(-x)dx.\]
For $n=0$ this makes sense if $x_0,x_1 \notin \bZ$, and one has  
\[ R_0(f)=\sum_{x_0\le j <x_1} f(j)-\int_{x_0}^{x_1} f(x)dx.\]
Integrating by parts and using $b_n^\prime=b_{n-1}$ we find
\[ R_n(f)=f^{(n-1)}(x_1)b_n(-x_1)-f^{(n-1)}(x_0)b_n(-x_0)+R_{n-1}(f).\]
Applying this iteratively and using the formula for $R_0(f)$ we obtain
\[ \sum_{x_0 \le j<x_1}f(j)=\int_{x_0}^{x_1} f(x)dx+\sum_{n=1}^{N+1} \big(f^{(n-1)}(x_0)b_n(-x_0)-f^{(n-1)}(x_1)b_n(-x_1)\big)+R_{N+1}(f).\]
By taking limits with $x_0$ or $x_1$ approaching integers one can check that this equation remains true for \emph{any} $x_0,x_1 \in \bR$, as long as the value $b_n(-x_0)$ is defined as in \eqref{eqn:berpol} (use that for $x_0\in \bZ$ and $0<\epsilon<1$ one has $-n!\cdot b_n(-x_0+\epsilon)=B_n(\{-x_0+\epsilon\})=B_n(\epsilon)\rightarrow B_n(0)=-n!\cdot b_n(-x_0)$ as $\epsilon \rightarrow 0^+$).  Taking the limit as $x_1 \rightarrow \infty$ we obtain
\begin{equation} 
\label{eqn:EulerMac3}
\sum_{x_0 \le j}f(j)=\int_{x_0}^\infty f(x)dx+\sum_{n=1}^{N+1} f^{(n-1)}(x_0)b_n(-x_0)+\int_{x_0}^\infty f^{(N+1)}(x)b_{N+1}(-x)dx.
\end{equation}
Let $\varphi$ be a Schwartz function.  Apply \eqref{eqn:EulerMac3} to $f(x)=\varphi(x/k)$ and $x_0=ks$.  After a change of variables $t=x/k$ in the integral, and using the formula \eqref{eqn:berpol}, we obtain
\begin{equation} 
\label{eqn:remainder}
\sum_{j \in \bZ,\, j \ge ks}\varphi\big(\tfrac{j}{k}\big) = k \int_s^\infty \varphi(t)dt-k\sum_{n=1}^N \frac{1}{k^n}\frac{B_n(\{-ks\})}{n!}\varphi^{(n-1)}(s)+\frac{1}{k^N}r,
\end{equation}
where $r$ is
\[ r=-\frac{B_{N+1}(\{-ks\})}{(N+1)!}\varphi^{(N)}(s)+\int_s^\infty \varphi^{(N+1)}(t)b_{N+1}(-kt)dt. \]
Using the fact that $b_{N+1}$ is a bounded function on $\bR$, it follows that $r$ is $O(1)$ hence $k^{-N}r$ is $O(k^{-N})$.  The first two terms in \eqref{eqn:remainder} lead to the asymptotic expansion \eqref{eqn:EulerMac1}.

The proof of the asymptotic expansion for $\zeta \ne 1$ a root of unity runs along the same lines, with $b_n$ replaced by $b_{n,\zeta}$.  No integral term appears in this case because $b_{0,\zeta}(x)=e^{-2\pi \i rx}\delta_{\bZ}(x)$ does not have a constant term.  The case involving a shift $\sigma$ can be handled similar to above, using also Taylor's theorem and the binomial expansion-like identity for a translation of a Bernoulli polynomial (see also the next section for a different approach).

The formulas for the full real line are similar: one takes the limit of equation \eqref{eqn:EulerMac3} as $x_0 \rightarrow \infty$ as well.  Since $f$ is Schwartz, the contributions from $f^{(n-1)}(x_0)$ will also go to zero, which is why the lower order terms in the expansion vanish identically.
\end{proof} 

\subsection{The Fourier transform of the Euler-Maclaurin formula.}\label{sec:FourierTrans}
Let $L=\bR_{\ge 0}+s$, $\sigma$, $\zeta$ be as in the previous subsection, and let $m(k,j)=\zeta^j [C_{L,\sigma}](k,j)$.  The proof of Lemma \ref{cor:1dcase} shows that the asymptotic expansion of $\Theta(m;k)$ is valid in the space of tempered distributions, hence we may take Fourier transforms of both sides and obtain an asymptotic expansion for the sequence of generalized functions $\wh{\Theta}(m;k)$.  In fact this leads to a different approach to proving Lemma \ref{cor:1dcase}, sketched below.  We also make an observation about the relationship between asymptotic expansions and Laurent expansion of the Fourier transform that is used in a later section.

The Fourier transform of $\Theta(m;k)$ is the generalized function of the variable $\omega \in \bR$ given by the weakly convergent sum
\begin{align} 
\label{eqn:Fourier1}
\wh{\Theta}(m;k,\omega)&=\sum_{j\ge ks+\sigma}\zeta^j e^{-\i\omega j/k}\nonumber \\
&=\zeta^{\lceil ks+\sigma\rceil}e^{-\i\omega s}e^{-\i \omega (\{-ks-\sigma\}+\sigma)/k}\sum_{j\ge 0} \zeta^j e^{-\i \omega j/k},
\end{align}
where $\i=\sqrt{-1}$.  Let $\varepsilon<0$ and introduce the complex variable $z=\omega+\i \varepsilon$.  Replacing $\omega$ by $z$ in the sum of equation \eqref{eqn:Fourier1}, one obtains a convergent geometric series.  Define
\begin{equation}
\label{eqn:thetamer}
\Phi_m(k,z)=\frac{e^{-\i z (\{-ks-\sigma\}+\sigma)}}{1-\zeta e^{-\i z}},\end{equation}
then $\wh{\Theta}(m;k,\omega)$ is the boundary value of $\zeta^{\lceil ks+\sigma\rceil}e^{-\i \omega s}\Phi_m(k,z/k)$, i.e. the weak limit
\begin{equation} 
\label{eqn:1dbdvalue}
\wh{\Theta}(m;k,\omega)=\zeta^{\lceil ks+\sigma\rceil}e^{-\i \omega s}\lim_{\varepsilon\rightarrow 0^-}\Phi_m(k,z/k).
\end{equation}
To obtain an asymptotic expansion in powers of $k$, we replace $\Phi_m(k,z)$ with its Laurent series $\Phi_m^{\tn{Laur}}(k,z)$ at $z=0$.  Using the generating functions \eqref{eqn:genfcn1}, \eqref{eqn:genfcn2}, one finds
\[ \Phi_m^{\tn{Laur}}(k,z)=\frac{1}{\i z}\sum_{n=0}^\infty \frac{B_{n,\zeta}(\{-ks-\sigma\}+\sigma)}{n!}(-\i z)^n.\]

Using Taylor's theorem one can argue that $\Phi_m^{\tn{Laur}}(k,(\omega+\i\varepsilon)/k)$ is an asymptotic expansion of the sequence of generalized functions $\Phi_m(k,(\omega+\i\varepsilon)/k)$ of $\omega$.  (In fact in this case for $\varepsilon<0$ and $0\ne z$ sufficiently small (away from the other poles of $\Phi_m(k,z)$) the Laurent series $\Phi_m^{\tn{Laur}}(k,z)$ really converges to $\Phi_m(k,z)$, an observation that we will use in a later section.)  Taking the boundary value term-by-term in this series we obtain\footnote{We will not justify the exchange of limits that is implicit here, since we gave an independent proof in the previous section.} an asymptotic expansion for $\wh{\Theta}(m;k,\omega)$:
\begin{equation} 
\label{eqn:1dfouriertransform}
\zeta^{\lceil ks+\sigma\rceil}e^{-\i \omega s}k\Big(\lim_{\varepsilon\rightarrow 0^-} \frac{B_{0,\zeta}}{\i \omega-\varepsilon}-\sum_{n=1}^\infty \frac{1}{k^n}\frac{B_{n,\zeta}(\{-ks-\sigma\}+\sigma)}{n!}(-\i s)^{n-1}\Big). 
\end{equation}
Taking the inverse Fourier transform term-by-term yields the expansions in Lemma \ref{cor:1dcase}.  Recall that the inverse Fourier transform of the boundary value $\lim_{\varepsilon\rightarrow 0^-}(\i\omega-\varepsilon)^{-1}$ is the Heaviside distribution.

\begin{remark}
\label{rem:laurentexpansion}
If $s=0$ observe that one has the following concise description of the asymptotic expansion of $\Theta(m;k)$ noted in \cite{BerlineVergneLocalAsym}: it is obtained by taking inverse Fourier transforms and boundary values term-by-term in the Laurent expansion at $z=0$ of a suitable meromorphic function whose boundary value is $\wh{\Theta}(m;k)$.
\end{remark} 

\subsection{The polynomial part.}\label{sec:polpart}
Alongside the proof of Theorem \ref{thm:AsymExp} in the next subsection, we will compute the leading term of the expansion in the case $m=q[C_{P,\sigma}]$ where $q \in \QPol(\bZ\oplus \Lambda)$.  The formula will involve the `polynomial part' of $q[C_{P,\sigma}]$.

As in \eqref{eqn:Defqg}, $q$ has a unique decomposition
\[ q(k,\lambda)=\sum_{g \in \Lambda^\wedge_{\bQ}} g^{-\lambda}q_g(k,\lambda) \]
where $q_g(k,-)$ is a polynomial function of $\lambda$ ($q_g$ is also quasi-polynomial in $k$ by Remark \ref{rem:QPolFamily}).  The coefficient \begin{equation}
\label{eqn:polpart1}
q_1(k,v)=:q_\pol(k,v) 
\end{equation}
of $1 \in \Lambda^\wedge_{\bQ}$ is called the \emph{polynomial part} of $q$.  Thus $q_\pol(k,v)$ is a function of $(k,v) \in \bZ \times V$, which is polynomial in $v$ and quasi-polynomial in $k$.

Let $A$ be a rational affine subspace and let $\sigma \in \Lambda \otimes \bQ$.  Define
\[ E_{A,\sigma}=\tn{aff}(C_{A,\sigma}) \]
to be the smallest affine subspace of $\bR \oplus V$ containing $C_{A,\sigma}$.  We will define the `polynomial part' of $q\upharpoonright E_{A,\sigma}$, denoted $(q\upharpoonright E_{A,\sigma})_{\pol}$, in a way that generalizes the definition of $q_\pol$ given in equation \eqref{eqn:polpart1} (the former definition will apply when $E_{A,\sigma}=\bR \oplus V \Leftrightarrow A=V$).

We analyse the set $A_k=(kA+\sigma)\cap \Lambda$ for $k \in \bZ$.  It is invariant under translation by $\tn{lin}(A)\cap \Lambda$.  It may happen that $A_k$ is empty for all $k$.

\begin{example}
Let $V=\bR^2 \supset \bZ^2=\Lambda$.  Let $A=\{(x_1,0)|x_1\in \bR\}$ and $\sigma=(0,\tfrac{1}{2})$.  Then $A_k$ is empty for all $k$.
\end{example}

So assume there exists $k_0$ such that $A_{k_0}$ is non-empty, and choose $\nu_0 \in A_{k_0}$.  Let $n_0$ be the smallest positive integer such that $n_0A\cap \Lambda$ is non-empty, and choose $\xi_0 \in A$ such that $n_0\xi_0 \in\Lambda$.  

If $k$ is such that $A_k \ne \emptyset$ then necessarily $k \in k_0+n_0\bZ$.  Indeed let $\nu \in A_k$, then $\nu-\nu_0 \in (k-k_0)A\cap \Lambda$, so $k-k_0$ is a multiple of $n_0$.

Define
\begin{equation} 
\label{eqn:lambdak}
\lambda_k=(k-k_0)\xi_0+\nu_0.
\end{equation}
Then $\lambda_k$ lies in the lattice $\Lambda$ if and only if $k \in k_0+n_0\bZ$, and in this case
\[ A_k=\lambda_k+\tn{lin}(A)\cap \Lambda.\]

One way to choose $\nu_0$ and $\xi_0$ is as follows.  Choose a complementary subspace $R$ to $\tn{lin}(A)$ in $V$ such that
\[ \Lambda=(\Lambda \cap \tn{lin}(A))\oplus (\Lambda \cap R).\]
Such an $R$ can always be found by lifting a lattice basis for the image of $\Lambda$ in the quotient space $V/\tn{lin}(A)$.  Then for any $\rho \in V_{\bQ}$ and $k \in \bZ$, the intersection $(kA+\rho)\cap R$ is a singleton, and is a lattice point if and only if $kA+\rho$ contains a lattice point.  Taking $(\rho=0,k=1)$ and $(\rho=\sigma,k=k_0)$ in turn, we obtain $\xi_0$, $\nu_0$:
\[ R \cap A=\{\xi_0\}, \qquad R\cap (k_0A+\sigma)=\{\nu_0\}\]
with the desired properties.  In this case $R\cap (kA+\sigma)=\{\lambda_k\}$.

Let $\I_{A,\sigma}=k_0+n_0\bZ \subset \bZ$.  Thus the characteristic function $[\I_{A,\sigma}](k)$ of $\I_{A,\sigma}$ is a quasi-polynomial function of $k$.  If $k \in \I_{A,\sigma}$ then the restriction $r$ of $q$ to $\{k\}\times A_k=E_{A,\sigma}\cap(\{k\}\times \Lambda)$ is a quasi-polynomial on the affine lattice $\{k\}\times A_k$ (modelled on $\tn{lin}(A)\cap \Lambda \subset \tn{lin}(A)$), hence has a unique decomposition  as in \eqref{eqn:Defqg}
\begin{equation} 
\label{eqn:defrg}
r(k,\lambda)=\sum_g g^{-(\lambda-\lambda_k)}r_g(k,\lambda),
\end{equation}
where $g \in (\Lambda \cap \tn{lin}(A))^\wedge_\bQ$ and $r_g(k,-)$ is a polynomial function on the affine subspace $(kA+\sigma)$ of $V$.  Notice that the $g=1$ term in \eqref{eqn:defrg} is independent of the choice of $\lambda_k$ (whereas the other $r_g$ change by a phase when $\lambda_k$ is chosen differently), and so is canonically determined.

\begin{definition}
\label{def:polpart}
The \emph{polynomial part} $(q\upharpoonright E_{A,\sigma})_{\pol}(k,v)$ of $q$ is a function on $E_{A,\sigma} \cap (\bZ \times V)$ that is polynomial in $\tn{lin}(A)$ directions and quasi-polynomial in $k$, and which depends only on the restriction $r$ of $q$ to $(\bZ \oplus \Lambda)\cap E_{A,\sigma}$.  Let $\I_{A,\sigma} \subset \bZ$ be the set of $k$ such that $(kA+\sigma)\cap \Lambda \ne \emptyset$.  For $k \notin \I_{A,\sigma}$ set $(q\upharpoonright E_{A,\sigma})_{\pol}(k,v)=0$.  Else $(q\upharpoonright E_{A,\sigma})_{\pol}(k,v)=r_1(k,v)$ where $r_1$ is defined by equation \eqref{eqn:defrg}.
\end{definition}
More concretely if
\[ q(k,\lambda)=\sum_{g \in \Lambda^\wedge_\bQ} g^\lambda q_g(k,\lambda), \]
then, since $(kA+\sigma)\cap \Lambda=\lambda_k+\tn{lin}(A)\cap \Lambda$ if $k \in \I_{A,\sigma}$ (and is empty otherwise), one has
\begin{equation} 
\label{eqn:polpart}
(q\upharpoonright E_{A,\sigma})_\pol(k,v)=[\I_{A,\sigma}](k)\sideset{}{'}\sum_g g^{\lambda_k}q_g(k,v)
\end{equation}
where the sum is over $g \in \Lambda^\wedge_\bQ$ that restrict to trivial characters on the sublattice $\tn{lin}(A)\cap \Lambda$.
\begin{example}
Let $V=\bR^2 \supset \bZ^2=\Lambda$ with coordinates $(\lambda_1,\lambda_2)$. Let $A=\{(x_1,0)|x_1\in \bR\}$ and $\sigma=0$.  Let $q(k,\lambda)=(-1)^{\lambda_2}$.  Then the polynomial part of $q$ is equal to zero.  But the polynomial part $(q\upharpoonright E_A)_\pol=1$.
\end{example}
\begin{example}
Let $V=\bR^2 \supset \bZ^2=\Lambda$ with coordinates $(\lambda_1,\lambda_2)$. Let $A=\{(x_1,\tfrac{1}{4})|x_1\in \bR\}$ and let $\sigma=(0,\tfrac{1}{2})$.  Then $A_k=(kA+\sigma)\cap \Lambda$ is non-empty if and only if $k=2+4n$.  Let $q(k,\lambda)=(-1)^{\lambda_2}+c$ where $c$ is a constant.  Then the restriction $q(2+4n,(\lambda_1,1+2n))=c-1$.  So the polynomial part is $c-1$ and is obtained as the restriction of two components of $q$.  They can add to a zero polynomial part.
\end{example}

\subsection{Polyhedrons and asymptotics.}\label{sec:GenCase}
Let $P$ be a polyhedron and $\tn{aff}(P)$ the smallest affine subspace containing $P$.  We will write $E_{P,\sigma}$ as short for $E_{\tn{aff}(P),\sigma}$.  Let $q(k,\lambda)$ be a quasi-polynomial of degree $d$ on $\bZ\oplus \Lambda$.  Since $P$ has non-empty interior as a subset of $\tn{aff}(P)$, the function $(q\upharpoonright E_{P,\sigma})_{\pol}$ is entirely determined by its restriction to $C_{P,\sigma}\cap (\bZ \oplus \Lambda)$.  Consider the function 
\[ k \mapsto (q\upharpoonright E_{P,\sigma})_\pol(k,kv+\sigma) \]
as a quasi-polynomial function of $k$ taking values in the space of polynomials of $v \in \tn{aff}(P)$.  Then it is of degree less or equal to $d$ (in $k$) and can be expanded
\begin{equation} 
\label{eqn:degdpart}
(q\upharpoonright E_{P,\sigma})_\pol(k,kv+\sigma)=\sum_{j=0}^d k^jp_j(k,v) 
\end{equation}
where $p_j(k,v)$ is a polynomial function of $v \in \tn{aff}(P)$ depending quasi-polynomially on $k$.  The degree $d$ part $p_d(k,v)$ can be zero.

As we explained in Section \ref{sec:disttheta}, Theorem \ref{thm:AsymExp} follows from the next theorem.
\begin{theorem}
\label{prop:genasymexp}
Let $P$ be a rational polyhedron and $\sigma \in \Lambda \otimes \bQ$.  Let $q(k,\lambda)$ be a quasi-polynomial function of degree $d$.  Then $\Theta(q[C_{P,\sigma}];k)$ admits an asymptotic expansion
\[ \Theta(q[C_{P,\sigma}];k)\sim k^{\dim(P)+d}\sum_{n=0}^\infty \frac{1}{k^n}\theta_n(k), \qquad \theta_n(k,v)=\sum_{\zeta \in U(1)_\bQ}\zeta^k r_{n,\zeta}(v)\]
where the distributions $r_{n,\zeta} \in \R(\{P\})$ (see Definition \ref{def:Rmod}), and $\zeta$ ranges over a finite set of roots of unity (the same for all $n$).  The $n=0$ term is
\begin{equation} 
\label{eqn:leadterm}
\theta_0(k,v)=p_d(k,v)\delta_P(v),
\end{equation}
where $p_d(k,v)$ is the function on $\tn{aff}(P)$ defined in \eqref{eqn:degdpart}.
\end{theorem}
\begin{proof}
Consider first the case $q(\lambda)=g^\lambda$.  We borrow an idea from \cite{BerlineVergneLocalAsym}.  Using the Brianchon-Gram decomposition, we may write $[P]$ as a finite signed sum of characteristic functions $[T]$ for tangent cones $T \subset \tn{aff}(P)$ to $P$.  It is possible to decompose $[T]$ further (non-canonically) into a signed sum of characteristic functions of translates of unimodular cones.  Thus in the end it is enough to consider the case where $P=s+P_0$ is the translation by some $s \in V_\bQ$ of a unimodular cone $P_0$, that is, $P_0$ is the set of non-negative linear combinations of a collection of vectors $\alpha_1,...,\alpha_{\dim(P)} \in \Lambda$ which form a lattice basis of $\Lambda \cap \tn{lin}(P)$.  

We assume that the image of $\sigma$ in the quotient $V/(\tn{lin}(P)+\bR\cdot s)$ is integral, since otherwise $(k\cdot P+\sigma)\cap \Lambda=(ks+\sigma+P_0)\cap \Lambda$ is empty for all $k$.  In this case we can find $\sigma_0 \in \Lambda$ such that $\sigma+\sigma_0 \in \tn{lin}(P)+\bR\cdot s$.  Using Proposition \ref{prop:module}, we may as well assume that $\sigma \in \tn{lin}(P)+\bR\cdot s$.

Suppose first that $s \in \tn{lin}(P)$, so $\sigma \in \tn{lin}(P)$ as well.  Let $s_j$ (resp. $\sigma_j$) be the components of $s$ (resp. $\sigma$):
\[ s=\sum_{j=1}^{\dim(P)} s_j \alpha_j, \qquad \sigma=\sum_{j=1}^{\dim(P)} \sigma_j \alpha_j. \]
The cone $ks+\sigma+P_0$ is the product of the cones $ks_j+\sigma_j+\bR_{\ge 0}$.  We then apply the 1-dimensional formula Lemma \ref{cor:1dcase} to each factor of the product.

In the general case, let $A=\tn{aff}(P)$ and choose a complementary subspace $R$ to $\tn{lin}(A)$ in $V$ such that $\Lambda=(\Lambda \cap \tn{lin}(A))\oplus (\Lambda \cap R)$, as in Section \ref{sec:polpart}.  For any $v \in V$, let $v_1,...,v_{\dim(P)}$ be the components of the projection of $v$ onto $\tn{lin}(A)$ along $R$ relative to the basis $\alpha_1,...,\alpha_{\dim(P)}$ of $\tn{lin}(A)$.  Then
\[ [C_{P,\sigma}\cap (\bZ\oplus \Lambda)](k,\lambda)=[(kP+\sigma)\cap \Lambda](\lambda)=[\I_{A,\sigma}](k)\prod_{j=1}^{\dim(P)} [(ks_j+\sigma_j+\bR_{\ge 0})\cap \bZ](\lambda_j). \]
The element $g \in \Lambda^\wedge_\bQ$ restricts to a character $\zeta_j \in U(1)_\bQ$ for the sublattice $\bZ e_j\simeq \bZ$.  We then apply the 1-dimensional formula Lemma \ref{cor:1dcase} to each factor of the product.

In the 1-dimensional formula, the leading term is $O(k)$ (if $\zeta=1$) or $O(k^0)$ (if $\zeta\ne 1$), hence the leading term of the $\dim(P)$-fold product is $O(k^{\dim(P)})$ at most.  If $g$ restricts to a non-trivial character on $\Lambda \cap \tn{lin}(P_0)$, then there is at least one $j$ such that $\zeta_j \ne 1$; in this case it follows from the 1-dimensional formula that the $O(k^{\dim(P)})$ term in the asymptotic expansion vanishes.  The expression for the $O(k^{\dim(P)})$ term follows from the 1-dimensional case, the definition of the polynomial part (cf. \eqref{eqn:polpart}), and the compatibility of the measures under forming products (see Definition \ref{def:measures}).

Finally, consider the more general case $q(k,\lambda)=g^\lambda h(k,\lambda)$ where $g \in \Lambda^\wedge_{\bQ}$ and $h(k,\lambda)$ has degree $d$ and is polynomial in $\lambda$.  Let $\Theta(g[C_{P,\sigma}];k,v)\sim \A(k,v)$.  By Proposition \ref{prop:module} $\Theta(q[C_{P,\sigma}];k,v)\sim h(k,kv)\A(k,v)$, which shows that an asymptotic expansion exists, and moreover the coefficient of $k^{\dim(P)+d}$ is
\begin{equation} 
\label{eqn:asym1}
p_d(k,v)\A_0(k,v),
\end{equation}
where $\A_0(k,v)$ is the coefficient of $k^{\dim(P)}$ in $\A(k,v)$, and $p_d$ is defined as in \eqref{eqn:degdpart} as the coefficient of $k^d$ for $k \mapsto h(k,kv)$ (since $h(k,kv)=h(k,kv+\sigma)+O(k^{d-1})$, it makes no difference whether we use $h(k,kv)$ or $h(k,kv+\sigma)$ here).
\end{proof}

In case $P=A$ is an affine subspace, similar arguments as in the proof of Theorem \ref{prop:genasymexp} give the following exact formula for the asymptotic expansion, based on the corresponding formula in the 1-dimensional case, Lemma \ref{cor:1dcase}, and Proposition \ref{prop:module}.
\begin{proposition}
\label{prop:fullasym}
Let $A$ be an affine rational subspace of $V$, let $\sigma \in \Lambda \otimes \bQ$, and let $q(k,\lambda)$ be a quasi-polynomial on $\bZ \oplus \Lambda$.  Then $\Theta(q[C_{A,\sigma}];k)\sim e^{-\partial_\sigma/k}\A(k)$, where
\[ \pushQED{\qed} 
\A(k,v):=k^{\dim(A)}(q\upharpoonright E_{A,\sigma})_\pol(k,kv+\sigma)\delta_A(v).\qedhere
\popQED \]
\end{proposition}

\section{Expansions in polarized cones and the kernel of $\A$}\label{sec:polcones}
According to Theorem \ref{thm:AsymExp} there is a map $m \mapsto \A(m)$, which associates to a piecewise quasi-polynomial function $m \in \S(\Lambda)$ the asymptotic series $\A(m;k)$ for the family of distributions $\Theta(m;k)$.  In this section we obtain a description of the kernel of $\A$ on the subspace of $\S(\Lambda)$ consisting of functions admitting a decomposition into quasi-polynomials on polarized cones (Definition \ref{def:polcones}).  Note that $\A$ certainly has a non-trivial kernel: for example, in Lemma \ref{lem:EulerMac} we saw that if $1\ne \zeta \in U(1)$ is a root of unity, $\Lambda=\bZ\subset \bR=V$ and $m(k,j)=\zeta^j$ then $\A(m)=0$.

For $\eta \in \Lambda$ and $\zeta \in U(1)_\bQ$ we define the finite difference operator $\nabla_\eta^\zeta=1-\zeta \tau_\eta$, or equivalently
\[ (\nabla_\eta^\zeta m)(k,\lambda)=m(k,\lambda)-\zeta m(k,\lambda-\eta).\]
Generalizing the 1-dimensional example above, we have the following.
\begin{proposition}
\label{prop:findiff}
Let $m \in \S(\Lambda)$.  Suppose there exists $0\ne \eta \in \Lambda$, $1\ne \zeta \in U(1)_\bQ$, and an $N \in \bZ_{>0}$ such that $(\nabla_\eta^\zeta)^N m=0$.  Then $\A(m)=0$.
\end{proposition}
\begin{proof}
Applying $\A$ we find
\[ 0=\A\big((\nabla_\eta^\zeta)^Nm\big)(k)=(1-\zeta e^{-\partial_\eta/k})^N\A(m)(k).\]
As $\zeta \ne 1$, the power series obtained by Taylor expansion of $(1-\zeta e^{-x})$ about $x=0$ is invertible.  Inverting this series and applying the corresponding operator to both sides yields $\A(m)=0$. 
\end{proof}

\begin{proposition}
\label{prop:kernel}
Let $P$ be a polyhedron invariant under translations by elements of the rational subspace $L_0$ and let $\sigma \in \Lambda_\bQ$.  Let $p=p(k,\lambda)$ be polynomial in $k,\lambda$, let $g \in \Lambda^\wedge_\bQ$ and set $q(k,\lambda)=g^\lambda p(k,\lambda)$.  If $g|_{\Lambda\cap L_0}$ is a non-constant character then
\[ m=q[C_{P,\sigma}] \in \ker(\A).\]
\end{proposition}
\begin{proof}
Since $g|_{\Lambda \cap L_0}$ is a non-constant character, we can find an $\eta \in \Lambda \cap L_0$ such that $g^\eta=\zeta \ne 1$.  Then
\[ (\nabla_\eta^\zeta q)(k,\lambda)=g^\lambda (\nabla^1_\eta p)(k,\lambda).\]
Since $p$ is polynomial in $\lambda$, $(\nabla^1_\eta)^N p=0$ for $N$ sufficiently large, hence $(\nabla_\eta^\zeta)^N q=0$.  Since $P$ is invariant under translation by $\eta$, $(\nabla_\eta^\zeta)^N m=0$ as well.  Now apply Proposition \ref{prop:findiff}.
\end{proof}

\subsection{Expansion into quasi-polynomials on polarized cones.}
It is convenient to introduce some additional notation and terminology for this section.
\begin{definition}
If $P$ is a polyhedral cone, we denote by $L_P$ the \emph{apex set} of $P$.  Then $L_P$ is an affine subspace and $\tn{lin}(L_P)$ is the \emph{lineality space} of $P$, the subspace of translations preserving $P$.  If $\P$ is a collection of cones in $V$ and $L$ is an affine subspace, we write $\P_L$ for the sub-collection consisting of all those cones whose apex set is $L$.
\end{definition}

\begin{definition}
Let $H$ be a closed half-space and let $P$ be a polyhedral cone.  We say $P$ is $H$-\emph{polarized} if $P \subset H$ and the intersection of $P$ with the boundary of $H$ coincides with the apex set $L_P$ of $P$.  If $\P_L$ is a collection of cones whose apex sets equal $L$, then we say that $\P_L$ is $H$-polarized if each $P \in \P_L$ is $H$-polarized.
\end{definition}

Throughout the remainder of this section fix $m \in \S(\Lambda)$ with decomposition
\begin{equation} 
\label{eqn:decomp}
m=\sum_{P\in \P} m_P, \qquad m_P=\sum_{\sigma \in \Sigma_P} q_{P,\sigma}[C_{P,\sigma}], \qquad q_{P,\sigma}(k,\lambda)=\sum_{g \in \Lambda^\vee_\bQ}g^{-\lambda}q_{P,\sigma,g}(k,\lambda) 
\end{equation}
where $q_{P,\sigma,g}(k,\lambda)$ is polynomial in $\lambda$ and quasi-polynomial in $k$, $\P$ is a collection of distinct rational polyhedra, and $\Sigma_P \subset \Lambda_\bQ$ is a collection of shift vectors.  Without loss of generality we also assume that $q_{P,\sigma}\upharpoonright C_{P,\sigma}\ne 0$.  Recall that it is also part of the definition that the large collection of polyhedra
\[ P+[0,1]\sigma, \qquad P \in \P, \sigma \in \Sigma_P \]
is locally finite in $V$.

\begin{definition}
\label{def:polcones}
A piecewise quasi-polynomial function $m \in \S(\Lambda)$ \emph{admits a decomposition into quasi-polynomials on cones} if $m$ has a decomposition as in \eqref{eqn:decomp} where $\P$ is a collection of distinct rational polyhedral cones.  We say that $m \in \S(\Lambda)$ \emph{admits a decomposition into quasi-polynomials on polarized cones} if furthermore the collection $\P$ is `polarized' in the following sense: for each proper affine subspace $L \subsetneq V$ that occurs as the apex set of some cone in $\P$, there is a half-space $H$ such that $\P_L$ is $H$-polarized.  (Note that the polarizations are allowed to vary as a function of the affine subspace $L$.)  The set of all $m$ that admit a decomposition into quasi-polynomials on polarized cones is denoted $\S_{\tn{pol}}(\Lambda)$.
\end{definition}

\begin{example}
\label{ex:nonpolexp}
Not every $m \in \S(\Lambda)$ admits a decomposition into quasi-polynomials on polarized cones.  Here is an example for $V=\bR^2$, similar to Example \ref{ex:growingshifts}.  For $n \in \bZ_{> 0}=\{1,2,3,...\}$ define 1-dimensional polyhedra $P_n$ and shift vectors $\sigma_n$, $\sigma_n^\prime$ by
\[ P_n=[-1,1]\times \{n\}, \qquad \sigma_n=(-n,-n), \quad \sigma_n^\prime=(n,-n).\]
This collection of polyhedra and shift vectors is allowed to appear in the decomposition \eqref{eqn:decomp} of an element of $\S(\Lambda)$ because the collection of parallelograms
\[ P_n+[0,1]\sigma_n, \quad P_n+[0,1]\sigma_n^\prime, \qquad n \in \bZ_{> 0} \]
is locally finite.  Suppose we attempt to decompose each characteristic function further into a signed sum of characteristic functions of cones.  Let us write $1_S$ for the characteristic function of $S\subset \bR$.  For the collection $(P_n,\sigma_n)$, $n \in \bZ_{> 0}$ we are forced to use the expansion to the left
\[ 1_{[-1,1]}=1_{(-\infty,1]}-1_{(-\infty,-1]}+1_{\{-1\}}. \]
(If we tried to use an expansion to the right then, after translating by $\sigma_n$, we would end up with a collection of cones that is not locally finite.)  For the same reason the collection $(P_n,\sigma_n^\prime)$, $n \in \bZ_{> 0}$ must be expanded to the right
\[ 1_{[-1,1]}=1_{[-1,\infty)}-1_{[1,\infty)}+1_{\{1\}}.\]
The resulting decomposition for the full collection $(P_n,\sigma_n)$, $(P_n,\sigma_n^\prime)$, $n \in \bZ_{> 0}$ is a locally finite decomposition into cones, but the cones are not polarized.
\end{example}

\begin{proposition}
\label{prop:quasioncones}
Every $m \in \S(\Lambda)$ admits a decomposition into quasi-polynomials on cones.
\end{proposition}
\begin{proof}
Recall that we assume that the collection of polyhedra
\[ P+[0,1]\sigma, \qquad P \in \P, \sigma \in \Sigma_P \]
is locally finite in $V$.  For each $P \in \P$, the characteristic function $[P]$ can be decomposed into characteristic functions supported on cones in many different ways.  The difficulty is therefore to choose decompositions such that the resulting big collection $\{(Q,\sigma)|Q \in \mathcal{Q},\sigma \in \Sigma_Q\}$ satisfies the local finiteness condition.  

Fix an inner product $\pair{-}{-}$ on $V$.  For each pair $(P,\sigma)$ let $\beta_{P,\sigma}$ be the nearest point in $P+[0,1]\sigma$ to the origin.  By the local finiteness condition there are at most finitely many pairs $(P,\sigma)$ such that $\beta_{P,\sigma}=0$, and we choose any expansion of these terms in cones.  For pairs $(P,\sigma)$ such that $\beta_{P,\sigma}\ne 0$ we use any expansion of $P$ into cones contained in the closed half-space
\[ H(\beta_{P,\sigma}):=\{v \in V|\pair{\beta_{P,\sigma}}{v}\ge \|\beta_{P,\sigma}\|\}.\]
Equivalently, choose $t_0 \in [0,1]$ such that $\beta_{P,\sigma} \in P+t_0\sigma$, decompose $P+t_0\sigma$ into cones contained in $H(\beta_{P,\sigma})$ and then translate by $-t_0\sigma$ to obtain a decomposition of $P$.

Let $P_{\sigma,i}$, $i=1,...,n_{P,\sigma}$ be the resulting collection of cones.  We claim that the large collection
\[ P_{\sigma,i}+[0,1]\sigma, \qquad P \in \P, \sigma \in \Sigma_P, i=1,...,n_{P,\sigma} \]
is locally finite.  

It suffices to show that $P_{\sigma,i}+[0,1]\sigma$ is contained in the half-space $H(\beta_{P,\sigma})$ since any bounded subset of $V$ will only intersect finitely many of these half-spaces.  For $t \in [0,1]$ let $f(t)$ be the distance from $0$ to $P+t[0,1]$ (this is a continuous but not necessarily differentiable function of $t$).  By construction $P_{\sigma,i}+t_0\sigma \subset H(\beta_{P,\sigma})$.  Consider three cases: (i) $t_0=0$, (ii) $t_0=1$, (iii) $t_0\in (0,1)$.  If $t_0=0$, then $P_{\sigma,i} \subset H(\beta_{P,\sigma})$ and since $f(t_0)$ is the minimum, we must have $\pair{\beta_{P,\sigma}}{\sigma}\ge 0$.  Hence for any $v+t\sigma \in P_{\sigma,i}+t\sigma$ we have
\[ \pair{\beta_{P,\sigma}}{v+t\sigma}=\pair{\beta_{P,\sigma}}{v}+t\pair{\beta_{P,\sigma}}{\sigma}\ge \|\beta_{P,\sigma}\| \]
since $\pair{\beta_{P,\sigma}}{\sigma}\ge 0$ and $\pair{\beta_{P,\sigma}}{v}\ge \|\beta_{P,\sigma}\|$.  Similar arguments work in the other two cases; note that in case (ii) (resp. (iii)) we have $P_{\sigma,i}+\sigma \subset H(\beta_{P,\sigma})$ and $\pair{\beta_{P,\sigma}}{\sigma}\le 0$ (resp. $P_{\sigma,i}+t_0\sigma \subset H(\beta_{P,\sigma})$ and $\pair{\beta_{P,\sigma}}{\sigma}=0$).
\end{proof}
\begin{remark}
Note that the half-space $H(\beta_{P,\sigma})$ appearing in the proof of Proposition \ref{prop:quasioncones} depends on both $P$ and $\sigma$, which is why the resulting collection of cones might fail to be polarized, as we saw in Example \ref{ex:nonpolexp}.
\end{remark}
\begin{proposition}
If $m$ is a finite linear combination of functions $q_{P,\sigma}[C_{P,\sigma}]$, then $m \in \S_{\tn{pol}}(\Lambda)$.
\end{proposition}
\begin{proof}
We argue as in the proof of Proposition \ref{prop:quasioncones}, except we replace $\beta_{P,\sigma}$ with $\beta_P$, defined to be the nearest point in $P$ to the origin.
\end{proof}

\subsection{A partial converse to Proposition \ref{prop:kernel}.}
In this subsection we prove the converse of Proposition \ref{prop:kernel} for $m \in \S(\Lambda)$ admitting a decomposition into quasi-polynomials on polarized cones:
\begin{theorem}
\label{prop:mainprop}
Let $m \in \S_{\tn{pol}}(\Lambda)$ with decomposition as in \eqref{eqn:decomp}, and suppose $\A(m)=0$.  If $q_{P,\sigma,g}\upharpoonright C_{P,\sigma}\ne 0$, then the element $g \in \Lambda^\vee_\bQ$ restricts to a non-constant character on $\Lambda \cap \tn{lin}(L_P)$.
\end{theorem}
We list some easy consequences of the theorem.
\begin{corollary}
\label{cor:kernelofA}
The kernel of $\A$ in $\S_{\tn{pol}}(\Lambda)$ consists of locally finite sums of functions $m$ such that there exists $0\ne \eta \in \Lambda$, $1\ne \zeta \in U(1)_\bQ$, and an $N \in \bZ$ with $(\nabla_\eta^\zeta)^N m=0$.
\end{corollary}
\begin{proof}
Such functions lie in the kernel of $\A$ by Proposition \ref{prop:findiff}.  That this is the entire kernel (on $\S_{\tn{pol}}(\Lambda)$) follows from Theorem \ref{prop:mainprop} and an argument as in the proof of Proposition \ref{prop:kernel}: since $g \in \Lambda^\vee_\bQ$ restricts to a non-constant character on $\Lambda \cap \tn{lin}(L_P)$, we may choose $\eta \in \Lambda \cap \tn{lin}(L_P)$ such that $\zeta:=g^{-\eta} \ne 1$.  Then $(\nabla_\eta^\zeta)^N(q_{P,\sigma,g}[C_{P,\sigma}])=0$ for $N$ sufficiently large.
\end{proof}
\begin{corollary}
If $m \in \S_{\tn{pol}}(\Lambda)$ and the support of $m$ does not contain any line, then $\A(m)=0$ implies $m=0$.
\end{corollary}
Recall that for $g \in \Lambda^\vee_\bQ$, $m \in \S(\Lambda)$ we defined $g\cdot m \in \S(\Lambda)$ by $(g\cdot m)(k,\lambda)=g^{\lambda}m(k,\lambda)$. We will prove a generalization of the next result in Section \ref{sec:unicity}.
\begin{corollary}
\label{cor:polarizedunicity}
Let $m \in \S_{\tn{pol}}(\Lambda)$ and suppose $\A(g\cdot m)=0$ for all $g \in \Lambda^\vee_\bQ$.  Then $m=0$.
\end{corollary}
\begin{proof}
Fix a decomposition of $m$ as in \eqref{eqn:decomp} and let $g \in \Lambda^\vee_{\bQ}$.  Note that
\[ (g\cdot m)(k,\lambda)=\sum_{P,\sigma,h} (g^{-1}h)^{-\lambda}q_{P,\sigma,h}(k,\lambda)[C_{P,\sigma}](k,\lambda),\]
hence the coefficient of $q_{P,\sigma,g}[C_{P,\sigma}]$ is $(g^{-1}g)^{-\lambda}=1$.  Since $\A(g\cdot m)=0$, Theorem \ref{prop:mainprop} says $q_{P,\sigma,g}\upharpoonright C_{P,\sigma}=0$.
\end{proof}

Theorem \ref{prop:mainprop} is proved in two stages below.  First we prove a separation result (Lemma \ref{lem:separation}), showing that if $\A(m)=0$ then the contribution to $\A(m)$ from each distinct apex set must vanish separately.  Then Lemma \ref{lem:Laurent} treats the case of a single apex set. 

\begin{lemma}
\label{lem:separation}
Let $m \in \S_{\tn{pol}}(\Lambda)$ with decomposition as in \eqref{eqn:decomp}.  If $\A(m)=0$ then for each affine subspace $L \subset V$,
\[ \sum_{P \in \P_L} \A(m_P)=0.\]
\end{lemma}
\begin{proof}
For any affine subspace $L$ appearing as the apex set of some cone in $\P$, define
\[ \A(m,\P_L):=\sum_{P \in \P_L} \A(m_P), \qquad \A(m,\P\setminus \P_L):=\A(m)-\A(m,\P_L).\]
Now fix an $L$ and let $\pi \colon V \rightarrow V/\tn{lin}(L)$ be the quotient map.  Let $H$ be a half-space such that $\P_L$ is $H$-polarized.  By induction we may assume that we have already shown that $\A(m,\P_{L^{\prime \prime}})=0$ for each $L^{\prime\prime}$ with $\dim(L^{\prime\prime})<\dim(L)$, hence
\begin{equation} 
\label{eqn:sumzero}
0=\A(m)=\A(m,\P_L)+\A(m,\P\setminus \P_L)=\sum_{P\in \P_L}\A(m_P)+\sum_{L^\prime}\sum_{P \in \P_{L^\prime}}\A(m_P)
\end{equation}
where the sum over $L^\prime$ ranges over affine subspaces $L^\prime\ne L$ with $\dim(L^\prime)\ge \dim(L)$ (in particular, $L^\prime \setminus L \ne \emptyset$).

We can find an $H$-polarized polyhedral cone $C$ with apex set $L$ containing $\cup \P_L$.  Indeed if $0 \in L$ we may take $C$ to be the sum $\Sigma_{P\in \P_L}P$, and more generally we translate to the origin, take the sum, and translate back.

Choose a point $w \in L$.  We may choose an open ball $B$ around $w$ such that 
\[ F \cap B \ne \emptyset \quad \Leftrightarrow \quad w \in F \]
for any closed face $F$ of any $P \in \P$.  The support of the asymptotic series $\A(m,\P_L)$ is contained in $C$.  Hence $\A(m,\P_L)|_{B\setminus C}=0$ and by equation \eqref{eqn:sumzero}, $\A(m,\P\setminus \P_L)|_{B\setminus C}=0$.  

We will argue that $\A(m,\P\setminus \P_L)|_{B\setminus C}=0$ implies $\A(m,\P\setminus \P_L)|_B=0$.  Indeed by Remark \ref{rem:germs}, it suffices to show that each face $F$ of each $P \in \P\setminus \P_L$ that meets $B$, also meets $B \setminus C$.  So let $P \in \P_{L^\prime}$ and let $F$ be a face of $P$ that meets $B$.  By our choice of $B$, we know that $w \in F$.  Let $v \in L^\prime \setminus L$.  Since $L^\prime \subset F$, $v \in F$ as well.  The projection $\pi(F)\subset V/\tn{lin}(L)$ thus contains the ray $R$ beginning at $\pi(v)$ and passing through $\pi(w)$ (note that since $v \in L^\prime \setminus L$, $\pi(v)\ne \pi(w)$).  The point $\pi(w)$ is the apex of the pointed cone $\pi(C)$.  Since the ray $R$ begins at a point $\pi(v)$ different from the apex of $\pi(C)$, it meets the complement of $\pi(C)$.  It follows that $F$ meets the complement of $C$, hence $F$ meets $B \setminus C$.

Thus $\A(m,\P\setminus \P_L)|_B=0$.  Then using equation \eqref{eqn:sumzero}, $\A(m,\P_L)|_B=0$.  Applying Remark \ref{rem:germs}, now to $\A(m,\P_L)$, we obtain $\A(m,\P_L)=0$.
\end{proof}

The next lemma completes the proof of Theorem \ref{prop:mainprop} by determining the kernel of $\A$ for piecewise quasi-polynomials supported on a collection of polarized cones with a common apex set.
\begin{lemma}
\label{lem:Laurent}
Let $H$ be a half-space and let $\P$ be a finite collection of $H$-polarized cones with a common apex set $L$.  Let $m$ be as in equation \eqref{eqn:decomp}, and suppose $\A(m)=0$.  If $q_{P,\sigma,g}\upharpoonright C_{P,\sigma}\ne 0$, then the element $g \in \Lambda^\vee_\bQ$ restricts to a non-constant character on $\Lambda \cap \tn{lin}(L)$. 
\end{lemma}
\begin{proof}
Consider first the case where $L=\{0\}$ and all the shift vectors $\sigma$ are integral.  The elements of $\P$ are $H$-polarized pointed cones, and passing to a further decomposition if necessary, we may assume they are unimodular.  By \eqref{eqn:decomp} and Proposition \ref{prop:module},
\[ \Theta(m;k,v)=\sum_{P,\sigma,g} q_{P,\sigma,g}(k,kv)\Theta(g^{-1}[C_{P,\sigma}];k,v).\]
We will take the Fourier transform.  We let $\xi$ denote the variable in $V^\ast$.  For $a \in V^\ast$ we write $a(\partial_\xi)$ for the directional derivative in $V^\ast$ in the direction $a$.  The map $a \mapsto a(\partial_\xi)$ from $V^\ast$ to constant coefficient differential operators extends to a ring homomorphism $p \mapsto p(\partial_\xi)$ from polynomials on $V$, $\tn{Pol}(V)\simeq \tn{Sym}(V^\ast)$, to constant coefficient differential operators on $V^\ast$.  Taking the Fourier transform,
\[ \wh{\Theta}(m;k,\xi)=\sum_{P,\sigma,g}q_{P,\sigma,g}(k,k\i \partial_\xi)\wh{\Theta}(g^{-1}[C_{P,\sigma}];k,\xi).\]
Let
\[ Q=\{ \xi \in V^\ast| \pair{\xi}{v}<0, \text{ for all } v \in \cup \P\}.\]
Since $\P$ is $H$-polarized, $Q$ is an open subset of $V^\ast$.  Let $\rho \in Q$ and set $z=\xi+\i \rho \in V^\ast \otimes \bC$.  By taking a product of 1-dimensional cases (see equations \eqref{eqn:thetamer}, \eqref{eqn:1dbdvalue}), one finds that the Fourier transform is the boundary value
\[ \wh{\Theta}(g^{-1}[C_{P,\sigma}];k,\xi)=\lim_{\rho \rightarrow 0} \Phi_{P,\sigma,g}(z/k),\] 
the limit only being taken over $\rho \in Q$, where
\[\Phi_{P,\sigma,g}(z):=\frac{g^{-\sigma}e^{-\i\pair{\sigma}{z}}}{\prod_{\alpha \in \bm{\alpha}_P}(1-g^{-\alpha}e^{-\i\pair{\alpha}{z}})} \]
and $\bm{\alpha}_P=\{\alpha_1,...,\alpha_{\dim(P)}\} \subset \Lambda \cap P$ is a set of positive generators for $P$ (that form a lattice basis of $\Lambda \cap \tn{lin}(P)$).  If we now define
\[ \Phi_m(k,z)=\sum_{P,\sigma,g}q_{P,\sigma,g}(k,\i \partial_z)\Phi_{P,\sigma,g}(z)\]
then
\begin{equation} 
\label{eqn:ThetaBdValue}
\wh{\Theta}(m;k,\xi)=\lim_{\rho\rightarrow 0} \Phi_m(k,z/k).
\end{equation} 

Let $\Phi_{P,\sigma,g}^{\tn{Laur}}(z)$ be obtained from $\Phi_{P,\sigma,g}(z)$ by replacing each factor $(1-g^{-\alpha}e^{-\i \pair{\alpha}{z}})^{-1}$ with its Laurent series in the complex variable $\pair{\alpha}{z}$, and replacing $e^{-\i\pair{\sigma}{z}}$ with its Taylor series; the result can be viewed as a formal series in $\pair{\sigma}{z}$, $\pair{\alpha}{z}$, $\pair{\alpha}{z}^{-1}$, where for each $n\ge -\dim(P)$ there are finitely many terms of degree $n$ with respect to the re-scaling action $z\mapsto tz$ (and there are no terms of degree less than $-\dim(P)$).  Using our analysis of the 1-dimensional case (Section \ref{sec:FourierTrans}), it follows that 
\[ \wh{\A}(g^{-1}[C_{P,\sigma}];k,\xi)=\lim_{\rho \rightarrow 0}\Phi_{P,\sigma,g}^{\tn{Laur}}(z/k).\]
If we now define
\[ \Phi_m^{\tn{Laur}}(k,z)=\sum_{P,\sigma,g}q_{P,\sigma,g}(k,\i \partial_z)\Phi_{P,\sigma,g}^{\tn{Laur}}(z)\]
then
\[ \wh{\A}(m;k,\xi)=\lim_{\rho \rightarrow 0}\Phi_m^{\tn{Laur}}(k,z/k).\]
If $\wh{\A}(m)=0$ then $\Phi_m^{\tn{Laur}}=0$.  The series $\Phi_m^{\tn{Laur}}(k,z)$ actually converges to $\Phi_m(k,z)$ for $z$ sufficiently small with $\pair{\alpha}{z}\ne 0$ for all $\alpha \in \cup_P \bm{\alpha}_P$.  Thus $\Phi_m(k,z)$ vanishes for $z$ in an open set, hence for all $z$. By equation \eqref{eqn:ThetaBdValue}, $\wh{\Theta}(m)=0$.

We now consider the general case.  Write $m=m_1+m_2$ where $m_1$ is the sum of all contributions in \eqref{eqn:decomp} such that $g|_{\Lambda \cap \tn{lin}(L)}$ is the trivial character $1$.  By Proposition \ref{prop:kernel}, $\A(m_2)=0$ hence $\A(m_1)=\A(m)=0$.  So we may as well assume $m=m_1$ from the beginning, and prove that $\A(m)=0$ implies $\Theta(m)=0$.

Choose a complementary subspace $L_2$ to $L_1:=\tn{lin}(L)$ such that
\[ \Lambda=\Lambda_1\oplus \Lambda_2, \qquad \Lambda_1=\Lambda \cap L_1,\quad \Lambda_2=\Lambda \cap L_2.\]
Without loss of generality we can assume all the $\sigma$ are in $L_2$, since a shift in the direction of $L_1$ will have no affect on the range of the summation in the definition of $\Theta(m;k)$.  The decomposition determines a decomposition of the torus $T=\Lambda_0^\vee$ into a product $T_1\times T_2$, where $T_i=\Lambda_i^\vee$.  So any $g \in T$ decomposes uniquely into a product $g=g_1g_2$, $g_i \in T_i$.  The condition that $g|_{\Lambda \cap \tn{lin}(L)}$ is the trivial character means that $g_1=1$.

If necessary we can pass to a finer lattice $\ti{\Lambda}_2$ containing all of the shift vectors as well as the point of intersection $\{s\}=L_2 \cap L$, and then extend the quasi-polynomials $q_{P,\sigma}$ by $0$ to $\Lambda_1\oplus\ti{\Lambda}_2$, as in Remark \ref{rem:changelattice}.  This will not change the distributions $\Theta(m;k)$ or $\A(m)$.  As $L$ now contains lattice points, we may as well assume that $0 \in L$, hence $L=L_1$.  Finally, passing to a further decomposition of $m$ if necessary, we may assume that all the $P$ are unimodular.  Hence each $P$ is of the form $L_1 \times P_2$, where $P_2=P\cap L_2$ is a unimodular pointed cone for $\Lambda_2=\Lambda \cap L_2$.  Having made these adjustments, we will assume that $\Lambda$, $L$, $m$ satisfied these conditions from the beginning. 

Expanding using the decomposition \eqref{eqn:decomp} of $m$, we have
\begin{align*}
\Theta(m;k,v)&=\sum_{P,\sigma,g} q_{P,\sigma,g}(k,kv)\Theta(g^{-1}[C_{P,\sigma}];k,v),\\
\A(m;k,v)&=\sum_{P,\sigma,g} q_{P,\sigma,g}(k,kv)\A(g^{-1}[C_{P,\sigma}];k,v).
\end{align*}
By assumption $q_{P,\sigma,g}$ vanishes unless $g_1=1$.  Using the product decomposition we have,
\[ \Theta(g^{-1}[C_{P,\sigma}])=\sum_{\lambda \in (k\cdot P+\sigma)\cap \Lambda} g^{-\lambda}\delta_{\lambda/k}=\Theta([C_{L_1}])\otimes \Theta(g_2^{-1}[C_{P_2,\sigma}]), \]
and
\[ \Theta([C_{L_1}])=\sum_{\lambda_1 \in \Lambda_1} \delta_{\lambda_1/k}, \qquad \Theta(g_2^{-1}[C_{P_2,\sigma}])=\sum_{\lambda_2 \in (k\cdot P_2+\sigma)\cap \Lambda_2} g_2^{-\lambda_2}\delta_{\lambda_2/k}. \]
For the asymptotic expansions,
\[ \A(g^{-1}[C_{P,\sigma}])=k^{\dim(L_1)}\delta_{L_1}\otimes \A(g_2^{-1}[C_{P_2,\sigma}]),\]
where we used the fact that $\A([C_{L_1}])=k^{\dim(L_1)}\delta_{L_1}$.

Taking the Fourier transform and using the Poisson summation formula we have
\begin{equation} 
\label{eqn:L1part}
\wh{\Theta}([C_{L_1}];k,\xi_1)=\sum_{\lambda_1 \in \Lambda_1}e^{-\i \pair{\lambda_1}{\xi_1}/k}=(2\pi k)^{\dim(L_1)}\sum_{\lambda_1^\ast\in \Lambda_1^\ast}\delta_{2\pi k\lambda_1^\ast}(\xi_1). 
\end{equation}
It follows that the distribution $\wh{\Theta}(m;k)$ is a sum of translates
\[ \wh{\Theta}(m;k)=\sum_{\lambda_1^\ast \in \Lambda_1^\ast} \tau_{2\pi k\lambda_1^\ast}\wh{\Theta}_0(m;k) \]
where $\wh{\Theta}_0(m;k)$ is defined just like $\wh{\Theta}(m;k)$ except only taking the term $\lambda_1^\ast=0$ in \eqref{eqn:L1part}.  It suffices to show that $\wh{\Theta}_0(m;k)=0$.  For the asymptotic expansion we have
\[ \wh{\A}([C_{L_1}];k,\xi_1)=(2\pi k)^{\dim(L_1)}\delta_0(\xi_1).\]

The Fourier transforms $\wh{\Theta}(g_2^{-1}[C_{P_2,\sigma}])$, $\wh{\A}(g_2^{-1}[C_{P_2,\sigma}])$ can be expressed in terms of boundary values of meromorphic functions, as in the discussion above for the special case $L=0$.  Let $z=\xi_2+\i \rho$, where $\rho$ takes values in the open subset $Q \subset L_2^\ast$ defined as before (using the cones $P_2=P\cap L_2$, for $P \in \P$).  We have
\[ \wh{\Theta}(g_2^{-1}[C_{P_2,\sigma}];k,\xi_2)=\lim_{\rho \rightarrow 0} \Phi_{P_2,\sigma,g_2}(z/k),\] 
the limit only being taken over $\rho \in Q$, where
\[\Phi_{P_2,\sigma,g_2}(z):=\frac{g_2^{-\sigma}e^{-\i\pair{\sigma}{z}}}{\prod_{\alpha \in \bm{\alpha}_{P_2}}(1-g_2^{-\alpha}e^{-\i\pair{\alpha}{z}})} \]
and $\bm{\alpha}_{P_2}=\{\alpha_1,...,\alpha_{\dim(P_2)}\} \subset \Lambda \cap P_2$ is a set of positive generators for $P_2$ (that form a lattice basis of $\Lambda \cap \tn{lin}(P_2)$).  Likewise
\[ \wh{\A}(g_2^{-1}[C_{P_2,\sigma}];k,\xi_2)=\lim_{\rho \rightarrow 0}\Phi_{P_2,\sigma,g_2}^{\tn{Laur}}(z/k),\]
where $\Phi_{P_2,\sigma,g_2}^{\tn{Laur}}$ is defined as before.

Define
\[ \Phi_m(k,\xi_1,z)=(2\pi k)^{\dim(L_1)}\sum_{P,\sigma,g}q_{P,\sigma,g}(k,\i k\partial_{\xi_1}+\i \partial_z)\Big(\delta_0(\xi_1)\Phi_{P_2,\sigma,g_2}(z)\Big)\]
which may be viewed as a meromorphic function of $z$ taking values in the space of distributions on $L_1^\ast$.  Likewise define
\[ \Phi_m^{\tn{Laur}}(k,\xi_1,z)=(2\pi k)^{\dim(L_1)}\sum_{P,\sigma,g}q_{P,\sigma,g}(k,\i k\partial_{\xi_1}+\i \partial_z)\Big(\delta_0(\xi_1)\Phi^{\tn{Laur}}_{P_2,\sigma,g_2}(z)\Big),\]
which may be viewed as a formal series (with finitely many terms of degree $n\ge -\dim(L_2)$) with coefficients that are distributions on $L_1^\ast$.  Then as before
\[ \wh{\Theta}_0(m;k,\xi)=\lim_{\rho\rightarrow 0} \Phi_m(k,\xi_1,z/k) \]
and
\[ \wh{\A}(m;k,\xi)=\lim_{\rho \rightarrow 0} \Phi_m^{\tn{Laur}}(k,\xi_1,z/k),\]
and arguing as above it follows from these descriptions that if $\wh{\A}(m)=0$ then $\wh{\Theta}_0(m)=0$.
\end{proof}

\section{A uniqueness result}\label{sec:unicity}
In this section we prove a generalization of Corollary \ref{cor:polarizedunicity} that applies to all $m \in \S(\Lambda)$.

\subsection{A filtration of $\S(\Lambda)$.}
We will make use of the following filtration on $\S(\Lambda)$.
\begin{definition}
\label{def:Sell}
For $0\le \ell \le \dim(V)$ let $\S_\ell(\Lambda)$ be the subspace of all $m \in \S(\Lambda)$ admitting an expansion as in \eqref{eqn:locfindecomp} with all polyhedra having dimension $\le \ell$.  For convenience we also define $\S_{-1}(\Lambda)=\{0\}$.
\end{definition}

\begin{lemma}
\label{lem:difflemma}
Let $P \subset V$ be a rational convex polyhedron with $\dim(P)=\ell$, and let $\sigma \in \tn{lin}(P)_\bQ$.  The difference $[C_P \cap (\bZ \oplus \Lambda)]-[C_{P,\sigma}\cap (\bZ \oplus \Lambda)] \in \S_{\ell-1}(\Lambda)$.  Moreover if $P$ is a cone with apex set $L$, then this difference is a signed sum of characteristic functions $[C_{P^\prime,\sigma^\prime} \cap (\bZ \oplus \Lambda)]$ where each $P^\prime$ is a cone of dimension $\le \ell-1$ with apex set $L$.
\end{lemma}
Before giving the general proof, let us explain the argument for $V=\bR \supset \bZ=\Lambda$ (see Figure \ref{fig:intropic2}).  Let $P \subset \bR$ be a closed interval.  Then we claim that for any $\sigma \in \bQ$, the characteristic functions $[C_{P,\sigma}\cap (\bZ \oplus \Lambda)]$ and $[C_P\cap (\bZ \oplus \Lambda)]$ are equal modulo $\S_0(\Lambda)$.  For example if $P=[a,\infty)$, $a \in \bQ$ and $\sigma \le 0$ then
\[ [C_{P,\sigma}]=[C_P]+\sum_i [C_{\{a\},\sigma_i}] \quad \text{on } \bZ \oplus \Lambda, \]
where the sum is over the finite set of $\sigma_i \in \bQ$, $\sigma<\sigma_i < 0$ such that the line $t \in \bR \mapsto (t,ta+\sigma_i) \in \bR^2$ has non-trivial intersection with the lattice $\bZ \oplus \Lambda=\bZ^2$.  For $P=[a,b]$ one has a similar expression
\[ [C_{P,\sigma}]=[C_P]+\sum_i [C_{\{a\},\sigma_i}]-[C_{\{b\}}]+[C_{\{b\},\sigma}]-\sum_j [C_{\{b\},\sigma^\prime_j}] \quad \text{on } \bZ \oplus \Lambda. \]
One can check this using, for example, the formula $1_{[a,b]}=1_{[a,\infty)}-1_{[b,\infty)}+1_{\{b\}}$ and the formula for a semi-infinite interval above.
\begin{proof}[Proof of Lemma \ref{lem:difflemma}.]
Using the Brianchon-Gram decomposition and then a further decomposition into simple cones (non-canonical), we can express $[P]$ as a signed sum of characteristic functions of affine simple cones of various dimensions $\le \ell$.  This gives a decomposition of $[C_P]$ (resp. $[C_{P,\sigma}]$) into a signed sum of characteristic functions of the form $[C_Q]$ (resp. $[C_{Q,\sigma}]$) with $Q$ an affine simple cone in $V$ of dimension $\le \ell$.  The cones in the decomposition of dimension $<\ell$ correspond to elements in $\S_{\ell-1}(\Lambda)$, so can be neglected.  Thus without loss of generality we can assume from the beginning that $P$ is an affine simple cone of dimension $\ell$.


As $P$ is a simple cone, we can find half-spaces $H_1,...,H_\ell$ of $V$ such that $P=W \cap H_1 \cap \cdots \cap H_\ell$, where $W=\tn{aff}(P)$.  Thus there is a product decomposition $[C_{P,\sigma}]=[C_W][C_{H_1,\sigma}]\cdots [C_{H_\ell,\sigma}]$ (note $\sigma \in \tn{lin}(P)$ implies $C_W=C_{W,\sigma}$).  By adding and subtracting terms, the difference $[C_P]-[C_{P,\sigma}]$ can be written as the sum over $1\le i\le \ell$ of products
\begin{equation} 
\label{eqn:prod}
[C_W][C_{H_1}]\cdots [C_{H_{i-1}}]([C_{H_i}]-[C_{H_i,\sigma}])[C_{H_{i+1},\sigma}]\cdots [C_{H_\ell,\sigma}].
\end{equation}
It suffices to consider each of these products separately, so fix $i \in \{1,...,\ell\}$, write $H=H_i$ and let $\partial H$ be the boundary of $H$.  Choose $(\mu, c)\in V^\ast \times \bR$ defining $H$, i.e. $H=\{v \in V|\pair{\mu}{v}+c\ge 0\}$.  Replacing $\mu$, $c$ by scalar multiples if necessary, we may assume $\mu \in \Lambda^\ast$ and $c, \pair{\mu}{\sigma} \in \bZ$.  If $\pair{\mu}{\sigma}=0$ then $[C_{H}]-[C_{H,\sigma}]=0$.  Else $\pair{\mu}{\sigma}$ is either positive or negative, say positive without loss of generality.  Then 
\[ C_H=\{(t,v)\in \bR_{>0}\times V|\pair{\mu}{v}+tc\ge 0\}, \quad C_{H,\sigma}=\{(t,v)\in \bR_{>0}\times V|\pair{\mu}{v-\sigma}+tc\ge 0\},\]
and so the difference $[C_H]-[C_{H,\sigma}]$ is the characteristic function of the set
\[ R=\{(t,v)\in \bR_{>0}\times V|0\le \pair{\mu}{v}+tc<\pair{\mu}{\sigma}\}.\]
If $(k,\lambda) \in R \cap (\bZ \oplus \Lambda)$ then $\pair{\mu}{\lambda}+kc \in \bZ$ with $0\le \pair{\mu}{\lambda}+kc < \pair{\mu}{\sigma}$.  Therefore
\[ [C_{H}\cap (\bZ \oplus \Lambda)]-[C_{H,\sigma}\cap(\bZ \oplus \Lambda)]=\sum_{s=1}^{\pair{\mu}{\sigma}-1}[C_{\partial H,s\ol{\sigma}}\cap (\bZ \oplus \Lambda)], \qquad \ol{\sigma}:=\tfrac{\sigma}{\pair{\mu}{\sigma}},\]  
and the result follows.

Now suppose $P$ is a cone with apex set $L$.  Choose a complementary subspace $L^\prime$ to $\tn{lin}(L)$ in $V$ such that $\Lambda=\Lambda^\prime \oplus(\Lambda \cap \tn{lin}(L))$, where $\Lambda^\prime=\Lambda \cap L^\prime$.  Let $P^\prime=P\cap L^\prime$, a pointed cone, and note that $P=P^\prime \times \tn{lin}(L) \subset L^\prime\times \tn{lin}(L)\simeq V$.  We may assume $P^\prime$ is simple, otherwise replace $[P^\prime]$ with a signed sum over characteristic functions of simple cones have the same apex as $P^\prime$.  Without loss of generality we may assume $\sigma \in L^\prime$, since $P$ is invariant under translations by $\tn{lin}(L)$.  Then $C_{P,\sigma}\simeq C_{P^\prime,\sigma}\times \tn{lin}(L) \subset \bR \times L^\prime\times \tn{lin}(L)\simeq \bR \times V$.  Examining the output of the argument above applied to $P^\prime$, one verifies that $[C_{P^\prime}\cap (\bZ \oplus \Lambda^\prime)]-[C_{P^\prime,\sigma}\cap (\bZ\oplus \Lambda^\prime)]$ can be expressed as a signed sum of characteristic functions $[C_{Q,\alpha}\cap (\bZ \oplus \Lambda^\prime)]$ where $Q$ is a cone of lower dimension and with the same apex as $P^\prime$.  The second claim now follows.
\end{proof}

\subsection{The tangent cone map for piecewise quasi-polynomial functions.}

\begin{definition}
For $v \in V_\bQ$, let $\S_v(\Lambda)$ denote the subspace of $\S(\Lambda)$ consisting of all $m$ admitting a decomposition
\[ m=\sum_{P,\sigma} q_{P,\sigma}[C_{P,\sigma}] \]
where each $P$ is a cone containing $v$ in its apex set.  (In particular the sum is finite.) 
\end{definition}

Recall that for $m \in \S(\Lambda)$ with decomposition
\[ m=\sum_{P\in \P,\sigma \in \Sigma_P}q_{P,\sigma}[C_{P,\sigma}],\]
the collection $\P$ of polyhedra is locally finite.
\begin{definition}
\label{def:prelim}
Let $P$ be a polyhedron.  For $v \in V$, let $T_vP$ be the tangent cone to $P$ at $v$ ($T_vP=\emptyset$ if $v \notin P$).  For $m \in \S(\Lambda)$ and $v \in V$ define $T_vm \in \S_v(\Lambda)$ by
\[ T_vm=\sum_{P,\sigma}q_{P,\sigma}[C_{T_vP,\sigma}] \]
using any decomposition of $m$.  We show that this is well-defined (independent of the choice of decomposition) in the next proposition.
\end{definition}

\begin{proposition}
\label{prop:tangentconemap}
$T_vm$ is the unique element of $\S_v(\Lambda)$ with the following property: there exists an open neighborhood $B$ of $v$ in $V$ and an integer $N>0$ such that $T_vm(k,\lambda)=m(k,\lambda)$ for $(k,\lambda) \in C_B \cap \{k>N\}$.
\end{proposition}
\begin{proof}
Given a decomposition of $m$, it is clear that the element $T_vm \in \S_v(\Lambda)$ in Definition \ref{def:prelim} has the property in the statement.  

To prove uniqueness, suppose $m_1,m_2 \in \S_v(\Lambda)$ both have the property in the statement, and set $m=m_1-m_2$. Then there is a $B$, $N$ such that $m(k,\lambda)=0$ for $(k,\lambda) \in C_B\cap \{k>N\}$.  Suppose $m \in \S_\ell(\Lambda) \cap \S_v(\Lambda)$; we will argue that $m \in \S_{\ell-1}(\Lambda)\cap \S_v(\Lambda)$ and then conclude by induction that $m=0$.  

Choose a decomposition of $m$
\[ m=\sum_{P \in \P, \sigma \in \Sigma_P} q_{P,\sigma}[C_{P,\sigma}],\]
where each $P$ is a cone of dimension $\le \ell$ containing $v$ in its apex set.  Using inclusion-exclusion formulas, we may also assume that the relative interiors of the polyhedra $P$ with $\dim(P)=\ell$ are disjoint.  Fix such a $P$ and let $\sigma \in \Sigma_P$.  By Lemma \ref{lem:difflemma}, modulo $\S_{\ell-1}(\Lambda)\cap \S_v(\Lambda)$ we may assume that $\sigma$ is the only element in $\Sigma_P\cap (\sigma+\tn{lin}(P))$.  Arguing as in Remark \ref{rem:crosssec}, the intersection
\[ C_B \cap \{k>N\} \cap C_{P,\sigma} \]
contains a set of the form $C_{R,\sigma} \cap \{k>K\}$, where $R$ is a relatively open subset of $P$.  Note that $C_{R,\sigma}$ is automatically disjoint from $C_{P,\sigma^\prime}$ when $\sigma^\prime+\tn{lin}(P)\ne \sigma+\tn{lin}(P)$.  By assumption the relative interior of $P$ is disjoint from the other cones, so arguing as in Remark \ref{rem:crosssec} and taking $K$ larger if necessary, the set $C_{R,\sigma} \cap \{k>K\}$ will not intersect $C_{P^\prime,\sigma^\prime}$ for any $P^\prime \ne P$, $\sigma^\prime \in \Sigma_{P^\prime}$.  Thus $q_{P,\sigma}(k,\lambda)=0$ for $(k,\lambda) \in C_{R,\sigma} \cap \{k>K\}$, and since $R$ is relatively open in $P$ it follows that $q_{P,\sigma}(k,\lambda)=0$ for all $(k,\lambda)\in C_{P,\sigma}$.  Hence the corresponding term in the decomposition of $m$ vanishes.  As $\sigma \in \Sigma_P$ was arbitrary and as $P \in \P$ was an arbitrary cone of dimension $\ell$, we conclude that $m \in \S_{\ell-1}(\Lambda)\cap \S_v(\Lambda)$.
\end{proof}

\begin{proposition}
\label{prop:locvanish}
If $T_vm=0$ for all $v \in V$ then $m=0$.
\end{proposition}
\begin{proof}
We suppose $m \in \S_\ell(\Lambda)$ and argue inductively.  Using inclusion-exclusion formulas, we may choose a decomposition \eqref{eqn:decomp} of $m$ such that the relative interiors of the $\ell$-dimensional polyhedra are disjoint.  Fix such a polyhedron $P$ and let $\sigma \in \Sigma_P$.  Let $v$ be in the relative interior of $P$.  Then
\[ 0=T_vm=\sum_{\sigma \in \Sigma_P}q_{P,\sigma}[C_{\tn{aff}(P),\sigma}]. \]
If $\sigma_1+\tn{lin}(P)\ne \sigma_2 +\tn{lin}(P)$, then $C_{\tn{aff}(P),\sigma_1}$, $C_{\tn{aff}(P),\sigma_2}$ are disjoint, hence the sum above splits into sub-sums over the cosets of $\tn{lin}(P)$, each of which must vanish separately.  On the other hand by Lemma \ref{lem:difflemma}, modulo $\S_{\ell-1}(\Lambda)$ we may assume that each coset has a single element, i.e. if $\sigma \in \Sigma_P$ then $\Sigma_P\cap (\sigma+\tn{lin}(P))=\{\sigma\}$.  Hence modulo $\S_{\ell-1}(\Lambda)$ we obtain $q_{P,\sigma}[C_{\tn{aff}(P),\sigma}]=0$, and so the corresponding term in the decomposition of $m$ vanishes.  Since $\sigma \in \Sigma_P$, $P \in \P$ with $\dim(P)=\ell$ were arbitrary, we conclude that $m \in \S_{\ell-1}(\Lambda)$. 
\end{proof}

\begin{corollary}
\label{cor:largek}
Let $m_1,m_2 \in \S(\Lambda)$ and suppose that for every bounded $B \subset V$ there exists $K_B>0$ such that $m_1(k,\lambda)=m_2(k,\lambda)$ for $(k,\lambda)\in C_B \cap \{k>K_B\}$.  Then $m_1=m_2$.
\end{corollary}
\begin{proof}
By Proposition \ref{prop:tangentconemap}, $T_vm_1=T_vm_2$ for all $v \in V$.  By Proposition \ref{prop:locvanish}, $m_1=m_2$.
\end{proof}

\subsection{The tangent cone map for asymptotic series.}
In this section we complete the proof of the unicity theorem, Theorem \ref{thm:unicity}, advertised earlier.  Recall the space $\R$ of families of distributions (parametrized by $k \in \bZ_{>0}$) from Definition \ref{def:Rmod}.
\begin{definition}
Let $\R_v \subset \R$ be the subspace of elements $\psi$ admitting a decomposition as in equation \eqref{eqn:decomppsi},
\[ \psi(k)=\sum_F r_F(k)\delta_F \]
where $F$ runs over all faces of all polyhedra $P \in \P$, and each $P \in \P$ is a cone containing $v$ in its apex set.
\end{definition}

It is convenient to define a similar `tangent cone' map on asymptotic series.
\begin{definition}
For $\psi \in \R$ and $v \in V$ define $T_v\psi$ to be the unique element of $\R_v$ that agrees with $\psi$ on an open neighborhood of $v$ (for all $k$).  Choosing $\P$ such that $\psi \in \R(\P)$ and choosing a decomposition
\[ \psi(k)=\sum_F r_F(k)\delta_F, \]
with $F$ running over all faces of all $P \in \P$, then
\[ T_v \psi(k)=\sum_F r_F(k)\delta_{T_vF}.\]
For $m \in \S(\Lambda)$ the coefficient of $k^n$ in the asymptotic series $\A(m)$ is in $\R$.  Define $T_v\A(m)$ by applying $T_v$ to each of the coefficients.
\end{definition}

\begin{proposition}
$T_v\A(m)=\A(T_vm)$.
\end{proposition}
\begin{proof}
Recall that there is a $K>0$ and an open ball $B$ around $v$ such that $T_vm(k,\lambda)=m(k,\lambda)$ for $(k,\lambda) \in C_B\cap \{k>K\}$.  It follows that $\A(T_vm)$ agrees with $\A(m)$ on $B$.  By construction $T_v\A(m)$ agrees with $\A(m)$ on some open ball around $v$.  Thus $T_v\A(m)=\A(T_vm)$ on some open ball around $v$.  But both lie in the subspace $\R_v$, and so we must have $T_v\A(m)=\A(T_vm)$.
\end{proof}

\begin{theorem}
\label{thm:unicity}
Let $m \in \S(\Lambda)$ and suppose $\A(g\cdot m)=0$ for all $g$.  Then $m=0$.
\end{theorem}
\begin{proof}
Note that for any $v \in V$,
\[ \A(g\cdot(T_vm))=\A(T_v(g\cdot m))=T_v\A(g\cdot m)=0 \]
for all $g$.  Since $T_vm \in \S_v(\Lambda)$ is given by a finite sum, it admits a decomposition into quasi-polynomials on polarized cones.  By Corollary \ref{cor:polarizedunicity}, $T_vm=0$.  Applying Proposition \ref{prop:locvanish}, we get $m=0$.
\end{proof}

\section{Pushforward of piecewise quasi-polynomial functions}
Let $W \subset V$ be a rational subspace, and let
\[ \pi \colon V \rightarrow V/W=V^\prime \]
be the quotient map.  The image of $\Lambda$ under $\pi$ is a lattice $\Lambda^\prime=\Lambda/(\Lambda \cap W) \subset V^\prime$.
\begin{definition}
\label{def:pushm}
Let $m \in \S(\Lambda)$, fix a decomposition of $m$ as in \eqref{eqn:locfindecomp}, and let
\[ S=\bigcup_{P \in \P,\sigma \in \Sigma_P} (P+[0,1]\sigma).\]
If $\pi$ restricts to a proper map $\pi|_S \colon S \rightarrow V^\prime$, then we define a function $\pi_\ast m$ on $\bZ \oplus \Lambda^\prime$ by
\[ \pi_\ast m(k,\lambda^\prime)=\sum_{\lambda \in \pi^{-1}(\lambda^\prime)} m(k,\lambda).\]
\end{definition}

To prove that $\pi_\ast m \in \S(\Lambda^\prime)$, we will need some results on families of polytopes defined by linear inequalities with parameters.  We briefly recall these results here; see for example \cite{BerlineVergneAnalyticContinuation, SzenesVergneResidue} for further details and references.  Consider a list $(\omega_i)_{i=1}^N$ of elements of $\Lambda^\ast$.  For $b=(b_1,...,b_N) \in \bR^N$ define a convex polyhedron
\begin{equation} 
\label{eqn:defPb}
\bm{P}(b)=\{v \in V|\pair{\omega_i}{v}\le b_i,\, i=1,...,N\}.
\end{equation}
There exists a closed cone $\C \subset \bR^N$ such that $\bm{P}(b) \ne \emptyset$ if and only if $b \in \C$.  If there exists a $b \in \C$ such that $\bm{P}(b)$ is compact, then $\bm{P}(b)$ is compact for every $b \in \C$.  Moreover there is a decomposition of $\C$ into a finite collection of closed polyhedral cones $\C_\nu$ having non-empty interiors, such that $\bm{P}(b)$ does not `change shape' when $b$ varies in a fixed $\C_\nu$.  More precisely, when $b$ varies in the interior of $\C_\nu$:
\begin{itemize}
\item the polytope $\bm{P}(b)$ has the same number of vertices $\{s_1(b),s_2(b),...,s_{n_\nu}(b)\}$,
\item for $1\le j \le n_\nu$, there exists a cone $C_j \subset V$ such that the tangent cone to $\bm{P}(b)$ at $s_j(b)$ is the affine cone $s_j(b)+C_j$,
\item the coordinates of the vertex $s_j(b)$ are linear functions of the $b_i$ with rational coefficients.
\end{itemize}
\begin{example}
Here is a very simple example.  Let $b=(b_1,b_2,b_3) \in \bR^3$ and
\[ \bm{P}(b)=\{x \in \bR|x\le b_1, \, -x\le b_2, \, -x \le b_3 \}=[-b_2,b_1]\cap [-b_3,\infty).\]
The cone $\C$ in this case is
\[ \C=\{b \in \bR^3|b_1+b_2 \ge 0, \, b_1+b_3 \ge 0 \}.\]
It has a decomposition $\C=\C_1\cup \C_2$,
\[ \C_1=\{b \in \C|b_2-b_3 \ge 0 \}, \qquad \C_2=\{b \in \C|b_3-b_2 \ge 0\}.\]
For $b \in \C_1$ the vertices of $\bm{P}(b)$ are $\{-b_3, b_1\}$, while on $\C_2$ the vertices of $\bm{P}(b)$ are $\{-b_2,b_1\}$.
\end{example}

Let $q \in \QPol(\Lambda)$ be a quasi-polynomial function, and let $\bm{P}(b)$, $\C=\cup \C_\nu$ be as above.  It is known that for each $\nu$, the function
\begin{equation} 
\label{eqn:Qb}
Q(b)=\sum_{\lambda \in \bm{P}(b) \cap \Lambda} q(\lambda) 
\end{equation}
restricts to a quasi-polynomial function of $b \in \C_\nu \cap \bZ^N$.  This is proven in \cite[Theorem 3.8]{SzenesVergneResidue}, where in fact one can find residue formulas for $Q(b)$ that exhibit its piecewise quasi-polynomial behavior.  
\begin{remark}
In comparing the setting of \cite{SzenesVergneResidue} with the setting here, one must translate from the description of the families of polytopes $\bm{P}(b)$ given above to the corresponding families of `partition polytopes'.  This may be done as follows (see also the introduction of \cite{BerlineVergneAnalyticContinuation} for example).  Let $\bm{P}(b)$ be as above:
\[ \bm{P}(b)=\{v \in V|\pair{\omega_i}{v}\le b_i,\, i=1,...,N\},\]
and assume that the $\omega_i$ span $V^\ast$ (otherwise $\bm{P}(b)$ has no chance of being non-empty and compact).  Let $e_1,...,e_N$ denote the standard basis of $\bR^N$, and let $x_1,...,x_N \in (\bR^N)^\ast$ denote the dual basis.  Define an inclusion $\iota \colon V \hookrightarrow \bR^N$ by duality:
\[ \pair{x_i}{\iota(v)}=-\pair{\omega_i}{v}, \qquad i=1,...,N.\]
The embedding carries $\Lambda$ into a sublattice of $\bZ^N \cap \iota(V)$.  Choose a complement $V^\prime$ of $\iota(V)$ such that $\bZ^N=(\iota(V)\cap \bZ^N)\oplus (V^\prime \cap \bZ^N)$, and let $T \colon \bR^N \rightarrow V^\prime$ be projection along $\iota(V)$.  Define $\phi_i=T(e_i)$, and let $\Phi=(\phi_i)_{i=1}^N$ be the list of images of the standard basis vectors.  Then the polytope $\bm{P}(b)$ is isomorphic to the \emph{partition polytope}
\[ \scr{P}(\Phi,\phi(b)):=\left \{(a_1,...,a_N) \in (\bR_{\ge 0})^N\Bigg|\sum_{i=1}^N a_i \phi_i=\phi(b) \right\}, \quad \text{with} \quad \phi(b)=\sum_{i=1}^N b_i\phi_i \]
via the map $v \in \bm{P}(b) \mapsto \iota(v)+b \in \scr{P}(\Phi,\phi(b))$.  Also there is a sublattice $\Gamma \subset \bZ^N$ (independent of $b$) such that when $b \in \bZ^N$, the map $\bm{P}(b) \rightarrow \scr{P}(\Phi,\phi(b))$ carries $\Lambda \cap \bm{P}(b)$ into $\Gamma \cap \scr{P}(\Phi,\phi(b))$.
\end{remark}

Another proof that $Q(b)$ is quasi-polynomial (for $q$ polynomial, but the method generalizes) may be found in \cite{BerlineVergneAnalyticContinuation}, based on `continuity' of the Brianchon-Gram decomposition.  

\begin{proposition}
\label{prop:pushm}
Let $m \in \S(\Lambda)$, fix a decomposition of $m$ as in \eqref{eqn:locfindecomp}, and let
\[ S=\bigcup_{P \in \P,\sigma \in \Sigma_P} (P+[0,1]\sigma).\]
If $\pi$ restricts to a proper map $\pi|_S \colon S \rightarrow V^\prime$, then the function $\pi_\ast m$ belongs to $\S(\Lambda^\prime)$.
\end{proposition}
\begin{proof}
For each $P \in \P$ the image $P^\prime=\pi(P)$ is a rational polyhedron in $V^\prime$.  The properness assumption implies that the collection of polyhedra $\{\pi(P)+[0,1]\pi(\sigma)|P \in \P, \sigma \in \Sigma_P\}$ is locally finite.  It suffices to prove the result for $m=q[C_{P,\sigma}]$, where $q(k,\lambda)$ is quasi-polynomial.  For simplicity consider first the case where $\sigma=0$.

We may decompose $q(k,\lambda)$ into a sum of products of the form $j(k)q_1(\lambda)$ where $j(k)$, $q_1(\lambda)$ are quasi-polynomial functions, so we may as well assume $q(k,\lambda)=j(k)q_1(\lambda)$.  If the result holds for $m_1=q_1[C_P]$ then it also holds for $m$, since $\pi_\ast m(k,\lambda^\prime)=j(k)\pi_\ast m_1(k,\lambda^\prime)$.  Thus it suffices to consider $m$ of the form $q[C_P]$ where $q=q(\lambda)$.  The remainder of the proof consists in reducing our question to the known quasi-polynomial behavior of functions of the type $Q(b)$ as in \eqref{eqn:Qb}.

The polyhedron $k\cdot P \subset V$ may be described as the set of solutions $v \in V$ to a list of linear inequalities
\[ \pair{\alpha_i}{v} \le ka_i, \quad i=1,...,n_P \]
and, after multiplying both by a sufficiently large integer if necessary, we may assume $\alpha_i \in \Lambda^\ast$, $a_i \in \bZ$.  Let $\ol{\beta}_1,...,\ol{\beta}_n \in (\Lambda^\prime)^\ast$ be a lattice basis ($n=\dim(V^\prime)$), and $\beta_j=\pi^\ast \ol{\beta}_j$.  Then for any $v^\prime \in V^\prime$ the fibre
\[ \pi^{-1}(v^\prime)=\{v \in V|\pair{\beta_j}{v}=\pair{\ol{\beta}_j}{v^\prime}, j=1,...,n \},\]
and this can be expressed as an intersection of $2n$ half-spaces $H_{+,j}(v^\prime), H_{-,j}(v^\prime)$ with $j=1,...,n$, defined by
\[ H_{\pm,j}(v^\prime)=\{v \in V|\pm \pair{\beta_j}{v} \le \pm \pair{\ol{\beta}_j}{v^\prime}\}.\]
Let $(\omega_i)_{i=1}^N$ be the list of the $N=n_P+2n$ elements $\alpha_1,...,\alpha_{n_P},\beta_1,...,\beta_n,-\beta_1,...,-\beta_n$.  The convex polyhedron $\bm{P}(b)$, $b \in \bR^N$ determined by the list $(\omega_i)_{i=1}^N$ is defined as in \eqref{eqn:defPb}.  Define a map
\[ \ul{b} \colon \bR_{>0} \times V^\prime \rightarrow \bR^N,\] 
by
\[ \ul{b}(t,v^\prime)=(ta_1,...,ta_{n_P},\pair{\ol{\beta}_1}{v^\prime},...,\pair{\ol{\beta}_n}{v^\prime},-\pair{\ol{\beta}_1}{v^\prime},...,-\pair{\ol{\beta}_n}{v^\prime}).\]
By construction, we have
\[ k\cdot P \cap \pi^{-1}(\lambda^\prime)=\bm{P}(\ul{b}(k,\lambda^\prime)), \quad \text{and} \quad \ul{b}^{-1}(\C)=C_{\pi(P)},\]
where $\C$ is the set of $b \in \bR^N$ such that $\bm{P}(b) \ne \emptyset$.  The function $\pi_\ast m$ is the pullback, under $\ul{b} \colon \bZ_{>0} \times \Lambda^\prime \rightarrow \bZ^N$, of the function
\[ Q(b)=\sum_{\lambda \in \bm{P}(b)\cap \Lambda} q(\lambda).\]
Recall there is a decomposition $\C=\cup \C_\nu$ such that $Q(b)$ restricts to a quasi-polynomial function on each $\C_\nu$ (see the paragraph following \eqref{eqn:defPb}).  Being the pullback under $\ul{b}$, it follows that $\pi_\ast m$ is quasi-polynomial on each of the cones $\ul{b}^{-1}(\C_\nu)$.  

Any cone in $\bR \oplus V^\prime$ contained in $C_{\pi(P)}$ is of the form $C_R$ for some subset $R \subset V^\prime$.  As $\ul{b}^{-1}(\C_\nu) \subset \ul{b}^{-1}(\C)=C_{\pi(P)}$, there are rational polyhedra $P_\nu^\prime \subset P^\prime=\pi(P)$ such that $\ul{b}^{-1}(\C_\nu)=C_{P_\nu^\prime}$.  Thus we have shown that $\pi_\ast m$ restricts to a quasi-polynomial function on each of the cones $C_{P_\nu^\prime}$, which cover $C_{P^\prime}=C_{\pi(P)}$.  It now follows from Example \ref{ex:inclusionexclusion} that $\pi_\ast m \in \S(\Lambda^\prime)$.

Now consider more generally $m=q[C_{P,\sigma}]$ where $\sigma$ is now allowed to be non-zero.  If $\sigma \in \Lambda$ then we may write $m=\tau_\sigma(m_1)$ with $m_1=(\tau_{-\sigma}q)[C_P]$ and use $\pi_\ast(\tau_\sigma(m_1))=\tau_{\pi(\sigma)}(\pi_\ast m_1)$.  In general, we may pass to a finite index extension $\ti{\Lambda} \supset \Lambda$ containing $\sigma$, and extend $q(k,\lambda)$ by zero to $\bZ \oplus \ti{\Lambda}$ as in Remark \ref{rem:changelattice}.  Writing $\ti{m}$ for the extended function on $\bZ \oplus \ti{\Lambda}$, the argument above shows that $\pi_\ast \ti{m} \in \S(\pi(\ti{\Lambda}))$.  Since $\pi_\ast m$ is the restriction of $\pi_\ast \ti{m}$ to $\bZ \oplus \Lambda \subset \bZ \oplus \ti{\Lambda}$, we get the result.
\end{proof}

Here is a particular case used in \cite{ParVerSemiclassical}.  Let $\Lambda_0$, $\Lambda_1$ be lattices in $V_0$, $V_1$ respectively, and let
\[ V=V_0\oplus V_1, \quad \Lambda=\Lambda_0\oplus \Lambda_1, \quad \pi_i\colon V \rightarrow V_i. \]
Let $R$ be a closed rational polyhedral cone in $V$ containing the origin (in particular $C_R=\{(t,v)|t>0,v\in R\}$).  Let $m_0=q[C_P] \in \S(\Lambda_0)$ and suppose $\pi_1$ is proper as a map from $R \cap \pi_0^{-1}(P)$ to $V_1$.  Let $c$ be any quasi-polynomial function on $\Lambda$.  For $(\lambda,\mu)\in \Lambda_0\oplus \Lambda_1$ define
\[ m_1(k,\mu)=\sum_{(\lambda,\mu)\in R \cap \Lambda}c(\lambda,\mu)m_0(k,\lambda).\]
Then $m_1 \in \S(\Lambda_1)$, since $m_1=(\pi_1)_\ast m$ where $m(k,\lambda,\mu)=c(\lambda,\mu)q(k,\lambda)[C_Q](k,\lambda,\mu)$, and $Q=R\cap \pi_0^{-1}(P)$.

\begin{corollary}
\label{cor:pushfunctorial}
Let $m \in \S(\Lambda)$, fix a decomposition of $m$ as in \eqref{eqn:locfindecomp}, and let
\[ S=\bigcup_{P \in \P,\sigma \in \Sigma_P} (P+[0,1]\sigma).\]
If $\pi$ restricts to a proper map $\pi|_S \colon S \rightarrow V^\prime$, then
\begin{equation} 
\label{eqn:pushtheta}
\Theta(\pi_\ast m;k)=\pi_\ast \Theta(m;k).
\end{equation}
Moreover if $\Theta(m;k)\sim \A(m;k)$ is its asymptotic expansion, then
\begin{equation}
\label{eqn:pushasym}
\Theta(\pi_\ast m;k)\sim \pi_\ast \A(m;k).
\end{equation}
\end{corollary}
\begin{proof}
Equation \eqref{eqn:pushtheta} follows immediately from the definition of $\pi_\ast m$.  For all $k \in \bZ_{> 0}$, the support of $\Theta(m;k)$ is contained in the set $S$.  By assumption the map $\pi|_S \colon S \rightarrow V^\prime$ is proper.  Equation \eqref{eqn:pushasym} follows from \eqref{eqn:pushtheta} together with continuity of the pushforward map from distributions supported on $S$ to distributions on $V^\prime$.
\end{proof}

\begin{example}
We give an example of Proposition \ref{prop:pushm}.  Let $V=\bR^2 \supset \bZ^2=\Lambda$ and let $m=q[C_P]$ where
\[ q(k,a,b)=\tfrac{1}{4}(1-(-1)^a)(1-(-1)^{a-b}),\]
and
\[ P=\{(x,y) \in \bR^2|0\le x \le 2, \, -x \le y \le x \}.\]
For $\pi \colon (x,y) \in \bR^2 \mapsto y \in \bR$ we have
\[ \pi_\ast m(k,b)=\tfrac{1}{4}\sum_{0\le a \le 2k,-a\le b\le a}(1-(-1)^a)(1-(-1)^{a-b}).\]
If $0\le b \le 2k$ then the sum ranges over $a \in [b,2k]\cap \bZ$ and equals $(1+(-1)^b)(k/2-b/4)$.  If $-2k \le b \le 0$ then the sum is over $a \in [-b,2k]\cap \bZ$ and equals $(1+(-1)^b)(k/2+b/4)$.  Thus setting $P_0=\{0\}$, $P_1=[0,2]$, $P_2=[-2,0]$ we have $\pi_\ast m=q_0[C_{P_0}]+q_1[C_{P_1}]+q_2[C_{P_2}]$ where
\[ q_0(k,b)=-k, \quad q_1(k,b)=(1+(-1)^b)(k/2-b/4), \quad q_2(k,b)=(1+(-1)^b)(k/2+b/4).\]
\end{example}


\bibliographystyle{amsplain}

\begin{thebibliography}{10}

\bibitem{BerlineGetzlerVergne}
N.~Berline, E.~Getzler and M.~Vergne, \emph{Heat kernels and {D}irac operators}, Springer-Verlag, 1992.

\bibitem{BerlineVergneAnalyticContinuation}
N.~Berline and M.~Vergne, \emph{Analytic continuation of a parametric polytope
  and wall-crossing}, Configuration Spaces: Geometry, Combinatorics and
  Topology, CRM Series \textbf{14} (2012), 111--172.

\bibitem{BerlineVergneLocalAsym}
\bysame, \emph{Local asymptotic {E}uler-{M}aclaurin expansion for {R}iemann
  sums over a semi-rational polyhedron}, Configuration Spaces, Springer, 2016,
  pp.~67--105.

\bibitem{CohenNumber}
H.~Cohen, \emph{Number theory, {V}olume {II}: {A}nalytic and modern tools},
  vol. 240, Springer Science \& Business Media, 2008.

\bibitem{DahmenMicchelli}
W.~Dahmen and C.~Micchelli, \emph{The number of solutions to linear Diophantine equations and multivariate splines}, Trans. Amer. Math. Soc. \textbf{308} (1988), no.~2, 509--532.

\bibitem{GuilleminSternbergEulerMaclaurin}
V.~Guillemin and S.~Sternberg, \emph{Riemann sums over polytopes}, Ann. Inst.
  Fourier (Grenoble) \textbf{57} (2007), no.~7, 2183--2195.

\bibitem{Hormander1}
L.~H{\"o}rmander, \emph{The analysis of linear partial differential operators
  {I}, distribution theory and fourier analysis}, Grundlehren der
  mathematischen Wissenchaften, vol. 256, Springer Verlag, 1990.
  
\bibitem{LMVerlindeSums}
Y.~Loizides and E.~Meinrenken, \emph{The decomposition formula for Verlinde sums}, arXiv:1803:06684.
  
\bibitem{MaZhangTransEll}
X.~Ma and W.~Zhang, \emph{Geometric quantization for proper moment maps: the {V}ergne conjecture}, Acta Math. \textbf{212} (2014), no.~1, 11--57.

\bibitem{ParadanFormalI}
\bysame, \emph{Formal geometric quantization}, Ann. Inst. Fourier (Grenoble)
  \textbf{59} (2009), no.~1, 199--238.
  
\bibitem{ParadanFormalII}
\bysame, \emph{Formal geometric quantization {II}}, Pacific J. Math. \textbf{253} (2011), 169--211.

\bibitem{ParadanFormalIII}
P.-E. Paradan, \emph{Formal geometric quantization {III}, {F}unctoriality in
  the spin-c setting}, arXiv:1704.06034.
  
\bibitem{Paradan98}
\bysame, \emph{The moment map and equivariant cohomology with generalized coefficients}, Topology \textbf{39} (2000), 401--444.

\bibitem{ParadanRiemannRoch}
\bysame, \emph{Localization of the {R}iemann-{R}och character}, J. Fun. Anal. \textbf{187} (2001), no.~2, 442--509.

  
\bibitem{ParVerAsymptotic}
P.-E. Paradan and M.~Vergne, \emph{Asymptotic distributions associated to piecewise quasi-polynomials}, arXiv:1708.08283.


\bibitem{ParVerSemiclassical}
\bysame, \emph{The equivariant index of twisted {D}irac
  operators and semi-classical limits}, arXiv:1708.08226.
  
\bibitem{WittenNonAbelian}
\bysame, \emph{Witten non abelian localization for equivariant {K}-theory, and the {$[Q,R]=0$} theorem}, Memoirs of the Am. Math. Soc., 2019.

\bibitem{RudinFunc}
W.~Rudin, \emph{Functional analysis}, McGraw-Hill, Inc., New York, 1991.

\bibitem{SzenesVergneResidue}
A.~{Szenes} and M.~{Vergne}, \emph{Residue formulae for vector partitions and
  {E}uler--{M}aclaurin sums}, Adv. Applied Math. \textbf{30} (2003),
  no.~1-2, 295--342.

\bibitem{VergneGradedTodd}
M.~Vergne, \emph{The equivariant {R}iemann-{R}och theorem and the graded {T}odd
  class}, C.R. Math. \textbf{355} (2017), no.~5, 563--570.
  
\bibitem{VergneFormalAhat}
\bysame, \emph{Formal equivariant $\hat{A}$-class, splines and multiplicities of the index of transversally elliptic operators}, Izvestiya: Mathematics \textbf{80} (2016), no.~5.

\end{thebibliography}

\providecommand{\bysame}{\leavevmode\hbox to3em{\hrulefill}\thinspace}
\providecommand{\MR}{\relax\ifhmode\unskip\space\fi MR }
\providecommand{\MRhref}[2]{%
  \href{http://www.ams.org/mathscinet-getitem?mr=#1}{#2}
}
\providecommand{\href}[2]{#2}

\end{document}